\documentclass[10pt,leqno]{article}

\title{On Average Risk-sensitive Markov Control Processes}
\author{Yun Shen\thanks{Fakult\"at Elektrotechnik und Informatik, Technische Universit\"at Berlin,  Marchstr.~23, 10587, Berlin, Germany ({\tt yun@ni.tu-berlin.de}).} 
 \and Klaus Obermayer\thanks{Fakult\"at Elektrotechnik und Informatik, and Bernstein Center for Computational Neuroscience, Technische Universit\"at Berlin, Marchstr.~23, 10587, Berlin, Germany ({\tt oby@ni.tu-berlin.de}). The work of this author
and the first author was supported by the BMBF, Bernstein Fokus Lernen TP1, 01GQ0911.} \and Wilhelm Stannat\thanks{Institut f\"ur Mathematik, and Bernstein Center for Computational Neuroscience, Technische Universit\"at Berlin, Stra{\ss}e des 17.~Juni 136, 10623, Berlin, Germany ({\tt stannat@math.tu-berlin.de}). The work of this author is supported by the BMBF, FKZ01GQ1001B.}}

\usepackage{amsmath}
\usepackage{amsfonts}
\usepackage{mathrsfs}
\usepackage{enumerate}
\usepackage{amsthm}

\usepackage[margin=1.7in]{geometry}

\newtheorem{theorem}{\sc Theorem}[section]
\newtheorem{lemma}[theorem]{\sc Lemma}
\newtheorem{proposition}[theorem]{\sc Proposition}
\newtheorem{definition}[theorem]{\sc Definition}
\newtheorem{corollary}[theorem]{\sc Corollary}

\theoremstyle{remark}
\newtheorem{remark}[theorem]{{\it Remark}}
\newtheorem{assumption}[theorem]{{\it Assumption}}

\newtheorem{example}[theorem]{{\it Example}}

\newcommand{\diff}{\textrm{d}}
\newcommand{\bs}{\boldsymbol}

\newcommand{\lya}[1]{(A#1)}
\newcommand{\os}[1]{(C#1)}
\newcommand{\ue}[1]{(B#1)}

\date{July 22, 2015}

\usepackage[sort, numbers]{natbib}
\begin{document}
 \maketitle
 
\begin{abstract}
We introduce the Lyapunov approach to optimal control problems of average risk-sensitive Markov control processes with general risk maps. Motivated by applications in particular to behavioral economics, we consider possibly non-convex risk maps, modeling behavior with mixed risk preference. We introduce classical objective functions to the risk-sensitive setting and we are in particular interested in optimizing the average risk in the infinite-time horizon for Markov Control Processes on general, possibly non-compact, state spaces allowing also unbounded cost. Existence and uniqueness of an optimal control is obtained with a fixed point theorem applied to the nonlinear map modeling the risk-sensitive expected total cost. The necessary contraction is obtained in a suitable chosen seminorm under a new set of conditions: 1) Lyapunov-type conditions on both risk maps and cost functions that control the growth of iterations, and 2) Doeblin-type conditions, known for Markov chains, generalized to nonlinear mappings. In the particular case of the entropic risk map, the above conditions can be replaced by the existence of a Lyapunov function, a local Doeblin-type condition for the underlying Markov chain, and a growth condition on the cost functions.
\end{abstract}

{\small 
\vskip10pt
{\noindent \bf AMS.} 60J05, 93E20, 93C55, 47H07, 91B06}

\section{Introduction}
The average cost (or equivalently, ergodic cost) criterion is a popular infinite-time horizon criterion for optimizing stochastic dynamic systems that are typically modeled by \emph{Markov control processes} (MCPs, see, e.g., \cite{hernandez1996discrete, hernandez1999further}, and \cite{puterman1994markov} under the name \emph{Markov decision processes}). The research (see e.g., \cite{arapostathis1993discrete} and \cite[Chapter 10]{hernandez1999further} for a comprehensive survey) focuses on the expected long-run average cost, which is a risk-neutral criterion.

The aim of this paper is to consider a general risk-sensitive criterion. So far, most of the long-run average risk-sensitive (especially, risk-averse) control employs the exponential utility function \cite{howard1972risk,chung1987discounted, avila1998controlled, borkar2002risk, di2008infinite,coraluppi2000mixed,fleming1997risk,hernandez1996risk,jaskiewicz2007,
marcus1997risk, cavazos2010optimality}.

In recent years, pioneered by \citet{ruszczynski2010risk}, many authors (\cite{bauerle2013more,Shen2013, Cavus2014}) have developed a more general framework of risk-sensitive sequential control problems on Borel spaces by applying \emph{coherent/convex risk measures}, which were originally employed in mathematical finance by \citet{artzner1999coherent,follmer2002convex}, to the classical risk-neutral MCPs. In particular, the exponential utility function can be viewed as a special convex (but not coherent) risk measure, called \emph{entropic measure} \cite{follmer2002convex}. Among them, \citet{ruszczynski2010risk} considered coherent risk measures with the finite-horizon and discounted criteria, \citet{bauerle2013more} considered convex risk measures with bounded costs,  and \citet{Cavus2014} considered coherent risk measure with total undiscounted criterion for transient MCPs. In this paper, we shall apply a more general family of risk measures to ergodic MCPs on Borel state-action spaces equipped with possibly unbounded costs, and solve the optimization problem corresponding to the average risk-sensitive criterion.

In the same setting, \citet{Shen2013} introduced the concept of \emph{risk map} that generalizes the one-step conditional risk measure (see e.g., \cite{ruszczynski2010risk}). Weighted norm spaces were then applied to incorporate possibly unbounded costs and Lyapunov-type stability conditions that generalized known conditions for Markov chains were stated to ensure the existence of solutions to the optimality equation for the average risk-sensitive criterion. More specifically, to deal with unbounded costs and nonlinearity of risk maps, \cite{Shen2013} introduced a dominating coherent risk map (called \emph{upper module}) and assumed the existence of a Lyapunov function to it. However, given a non-coherent risk map, e.g., the entropic measure, this dominating risk map may be infinite for unbounded cost functions. Hence, a (real-valued) Lyapunov function need not exist and the theory developed in \cite{Shen2013} is reduced to be valid only for bounded costs.

In order to solve the above mentioned problem within the same framework, we introduce in this paper 1) constraints both on the cost function and risk map to control the growth of \emph{value iterations} and 2) additional minorization properties on bounded level-sets (small sets) that generalize the standard Doeblin condition. Under these conditions, we show the existence of a \emph{bounded forward invariant subset} that covers the whole iterations. Restricted to the bounded subset, we assume the existence of the Lyapunov function for the weaker type of dominating map, which is called \emph{upper envelope} in this paper, to ensure the existence of a unique solution to the optimality equation for the average-risk criterion. As a special case, we show that, when applying the entropic map, the above conditions are satisfied, if 1) a Lyapunov function exists for the entropic measure, 2) the local Doeblin condition holds for the underlying Markov chain, and 3) a growth condition for cost functions.


Most of the existing literature on risk-sensitive MCPs, especially that applies the entropic map, considers finite or countable state spaces (see, e.g., \cite{avila1998controlled, borkar2002risk,coraluppi2000mixed,fleming1997risk, hernandez1996risk,
marcus1997risk, cavazos2010optimality}), or bounded cost functions (see, e.g., \cite{bauerle2013more}). Comparing with the literature \cite{masi2000infinite, jaskiewicz2007, di2008infinite} of the same settings, i.e., Borel spaces and unbounded cost functions, we provide in this paper a more general framework which can be applied to all types of risk maps, and more importantly, with a conceptually simpler proof, whereas the methods developed in \cite{masi2000infinite, jaskiewicz2007, di2008infinite} can only be applied to the entropic map. Moreover, the conditions we state for the entropic map are simpler than the conditions stated in \cite{masi2000infinite, jaskiewicz2007, di2008infinite}.

The paper is organized as follows. In Section \ref{sec:model}, we briefly review the framework of MCPs and Lyapunov approach for ergodic MCPs with the weighted norm space, followed by an introduction of risk maps within the context of MCPs. In Section \ref{sec:obj}, the average risk criterion is formally defined and two sets of general conditions are stated to ensure the boundedness of iterations and the geometric contraction. We then prove the existence and uniqueness of a solution to the associated nonlinear Poisson equation, as well as the average risk criterion. As a special case, we show in Section \ref{sec:entropic} that under proper verifiable assumptions, the entropic map fits the theoretical framework developed in the previous section. Finally, in order to demonstrate the applicability of our theory, we show in Section \ref{sec:examples} that applied to a canonical MCP, various types of risk maps fit also the developed theoretical framework.


\subsubsection*{Notation}  Let $\mathsf X$ be a \emph{Borel space}, that is a Borel subset of a complete separable metric space, and $\mathcal B(\mathsf X)$ its Borel $\sigma$-algebra. Let $\mathsf X$ and $\mathsf Y$ be two Borel spaces. A \emph{stochastic kernel on $\mathsf X$ given $\mathsf Y$} is a function $\psi(\mathsf B|y), \mathsf B \in \mathcal B(\mathsf X), y \in \mathsf Y$ such that i) $\psi(\cdot|y)$ is a probability measure on $\mathcal B(\mathsf X)$ for every fixed $y \in \mathsf Y$, and ii) $\psi(\mathsf B|\cdot)$ is a measurable function on $\mathsf Y$ for every fixed $\mathsf B \in \mathcal B(\mathsf X)$. Let $\bar{\mathbb R} := \mathbb R \cup \{ \pm \infty \}$ be the extended real line. 
 
\section{Model and problem formulation}\label{sec:model}
\subsection{Markov control processes} \label{sec:mcps}
In this subsection, we briefly introduce the framework of Markov control processes, where we mostly follow the notation by \citet{hernandez1999further}. A Markov control process, $(\mathsf X, \mathsf A,\{\mathsf A(x)|x \in \mathsf X \},Q,c)$, consists of  the following components: \emph{state space} $\mathsf X$ and \emph{action space} $\mathsf A$, which are Borel spaces; the feasible action set $\mathsf A(x)$, which is a nonempty Borel space of $\mathsf A$, for a given state $x \in \mathsf X$; the \emph{transition model} $Q(\mathsf B|x,a), \mathsf B \in \mathcal B(\mathsf X), (x,a) \in \mathsf K$: a \emph{stochastic kernel} on $\mathsf X$ given $\mathsf K$, where $\mathsf K$ denotes the set of feasible state-action pairs $\mathsf K:=\{ (x,a)|x \in \mathsf X, a \in \mathsf A(x)\}$, which is a Borel subset of $\mathsf X \times \mathsf A$; and the \emph{cost function} $c$: $\mathsf K \rightarrow \mathbb R$, $\mathcal B(\mathsf K)$-measurable. Random variables are denoted by capital letters, whereas realizations of the random variables are denoted by lowercase letters. The process is assumed to be Markov, i.e., for each $t \in \mathbb N$, $\mathbb P(X_{t+1} \in \mathsf B|X_t=x, A_t =a, X_{t-1}, \ldots, X_0, A_0) = Q(\mathsf B|x,a)$, $\forall \mathsf B \in \mathcal B(\mathsf X)$ and $(x,a) \in \mathsf K$. 


We restrict in this paper to \emph{Markov policies}, $\boldsymbol \pi = [ \pi_0,\pi_1,\pi_2,\ldots]$, where each \emph{single-step policy} $\pi_t(\cdot |x_t)$, which denotes the probability of choosing action $a_t$ at $x_t$, $(x_t,a_t) \in \mathsf K$, is Markov (independent of the states and actions before $t$) and, therefore, a stochastic kernel on $\mathsf A$ given $\mathsf X$. We use the boldface to represent a sequence of policies while using the normal typeface for a single-step policy. Let $\Delta$ denote the set of all stochastic kernels on $\mathsf A$ given $\mathsf X$, $\mu$, such that $\mu(\mathsf A(x)|x) = 1$ and $\Pi_M := \Delta^\infty$ denotes the set of all Markov policies. A policy $f \in \Delta$ is \emph{deterministic} if for each $x\in \mathsf X$, there exists some $a\in \mathsf A(x)$ such that $f(\{a\}|x) = 1$. Let $\Delta_D \subset \Delta$ denote the set of all deterministic single-step policies. A policy $\boldsymbol \pi $ is said to be \emph{stationary}, if $\boldsymbol \pi = \pi^\infty$ for some $\pi \in \Delta$. For each $x \in \mathsf X$ and single-step policy $\pi \in \Delta$, define
\begin{align}
 c^\pi(x) := \int_{\mathsf A(x)} c(x,a) \pi(\diff a|x), P^\pi(\mathsf B|x) := \int_{\mathsf A(x)} Q(\mathsf B|x,a) \pi(\diff a|x), \mathsf B \in \mathcal B(\mathsf X).
 \label{eq:cpi}
\end{align}

The following \emph{average cost} (see, e.g., \citet{arapostathis1993discrete} and \citet[Chapter 10]{hernandez1999further}) is used as an objective: 
\begin{equation}
 S := \limsup_{T \rightarrow \infty} \dfrac{1}{T} S_T, \textrm{ where } S_T := \sum_{t=0}^T c(X_t,A_t). \label{eq:totalreward}
\end{equation}
The optimization problem is then to minimize the expected objective
\begin{equation}
 \inf_{\boldsymbol \pi \in \Pi_M} \mathbb E^{\boldsymbol \pi} \left[ S | X_0 = x\right] \label{eq:obj}
\end{equation}
by selecting a policy $\boldsymbol \pi$. We notice that due to the Markov property, the finite-stage objective function can be decomposed as follows,
\begin{align}
 \mathbb E^{\boldsymbol \pi}_{X_0} [ S_T  ] = c^{\pi_0}(X_0) +  \mathbb E^{\pi_0}_{X_0} [
   c^{\pi_1}(X_1) + \mathbb E^{\pi_1}_{X_1} [ & c^{\pi_2}(X_2) + \ldots + \mathbb
     E^{\pi_{T-1}}_{X_{T-1}} \left[ c^{\pi_T}(X_T) \right] \ldots  ]
 ], \label{eq:decompose}
\end{align}
where $ \mathbb E^{\pi_t}_{X_t}\left[ v(X_{t+1}) \right]$ denotes the \emph{conditional expectation} of the function $v$ of the successive state $X_{t+1}$ given current state $X_t$.

\subsection{Lyapunov approach for ergodic MCPs}
It is known that the optimization problem of the average cost criterion is closely related to the $w$-ergodicity of underlying Markov processes (see e.g., \cite[Chapter 10]{hernandez1999further}). This is usually established by finding a Lyapunov function $w$ with ``small'' level-sets \cite[Chapter 14]{meyn1993markov}. If the Lyapunov function is strong enough, the transition kernels converge exponentially fast towards the unique invariant measure. Among other variations \cite{kontoyiannis2005large,douc2004quantitative,del2003contraction}, \citet{hairer2011yet} stated a simplified version of the conditions, which is the main approach that we will follow and extend in this paper.

\subsubsection*{Weighted norm} We first introduce the weighted norm and seminorm. 
Let $w: \mathsf X \rightarrow [1,\infty)$ be a given real-valued $\mathcal B(\mathsf X)$-measur\-able function. Consider the $w$-norm
$$\lVert u \rVert_{w} := \sup_{x \in \mathsf X} \dfrac{\lvert u(x) \rvert }{w(x)}.$$
Let $\mathscr B_w$ be the space of real-valued $\mathcal B(\mathsf X)$-measurable functions with bounded $w$-norm. It is obvious that $\mathscr B \subset \mathscr B_w$, where $\mathscr B$ denotes the space of bounded $\mathcal B(\mathsf X)$-measur\-able functions.  Throughout this paper we call $w$ interchangeably a \emph{weight function} or a \emph{Lyapunov function}. One notable property is as follows: if $w_0 : \mathsf X \rightarrow [0,\infty)$ is a $\mathcal B(\mathsf X)$-measur\-able function, then $1+\beta w_0$ is a valid weight function for any positive coefficient $\beta$. Furthermore, $\mathscr B_{1+ \beta_1 w_0} = \mathscr B_{1+ \beta_2 w_0}$ for any two positive coefficients, $\beta_1$ and $\beta_2$. This property will be used several times throughout this paper.


For a given signed measure $\mu$ on $\mathcal B(\mathsf X)$ with $\int w\, \diff |\mu|
< \infty$, the integral $\int u\, \diff\mu$ is well-defined for all $u\in\mathcal B_w$. Let 
$$
\lVert \mu \rVert_w := \sup_{\lVert u \rVert_w \leq 1} \lvert \int_{\mathsf X} u \diff \mu \rvert
$$  
and note that 
$$
\lVert \mu \rVert_w = \int_{\mathsf X} w \diff \lvert \mu \rvert \geq \lVert \mu \rVert_{TV},
$$ 
where $\lVert \cdot \rVert_{TV}$ denotes the total variation norm. Let $\mathscr P$ be the space of all probability measures $\mu$ on $\mathcal B(\mathsf X)$ and $\mathscr P_w :=\{ \mu \in \mathscr P | \int w\, \diff|\mu|  < \infty \}$. For any $\mu \in \mathscr P_w$ and $f \in \mathscr B_w$, we introduce the following notation $$\mu[f] := \int_{\mathsf X} f(x) \mu(\diff x).$$

The following \emph{$w$-seminorm} is used throughout this paper:
\begin{align*}
 \lVert v \rVert_{s,w} := \sup_{x \neq y} \dfrac{\lvert v(x) - v(y) \rvert}{w(x)+w(y)}.
\end{align*}
When restricting to the space $\mathscr B$, i.e., setting $w \equiv 1$, the seminorm is called \emph{span-norm} in \cite{hernandez1989adaptive} or \emph{Hilbert seminorm} in \cite{gaubert2004perron}. The following lemma (see \cite[Lemma 2.1]{hairer2011yet}) establishes the connection between the weighted norm and seminorm. 
\begin{lemma}\label{lm:semi}
$\lVert v \rVert_{s,w} = \min_{c \in \mathbb R} \lVert v + c \rVert_w, \forall v \in \mathscr B_w$.
\end{lemma}

The partial ordering ``$\leq$'' between elements in $\mathscr B_w$ is defined as follows: we say $v \leq u$ if $v(x) \leq u(x)\ \forall x \in \mathsf X$. A real number $u \in \mathbb R$ can be viewed as a constant-valued function which belongs also to $\mathscr B_w$.

\subsubsection*{Ergodicity conditions}
We adapt the sufficient conditions for ergodicity of Markov chains stated by \citet{hairer2011yet} to the MCP-framework.
\begin{assumption}\label{assp:standard}
 \rm \begin{description}
 \item[\rm (i)] There exists a function $w: \mathsf X \rightarrow [0,\infty)$, which is $\mathcal
  B(\mathsf X)$-measurable, and constants $K \geq 0$ and $\gamma \in (0,1)$ such that
 \begin{align*}
  \int Q(\diff y | x,a)w(y) \leq \gamma w(x) + K, \forall (x,a) \in \mathsf K
 \end{align*}
 \item[\rm (ii)] There exists a constant $\alpha \in (0,1)$ and a probability measure $\mu$ so that 
 \begin{align*}
  Q(\mathsf C|x,a) \geq \alpha \mu(\mathsf C), \forall \mathsf C \in \mathcal B(\mathsf X), x \in \mathsf B, a \in \mathsf A(x)
 \end{align*}
with $\mathsf B := \{x \in \mathsf X: w(x) \leq R \}$ for some $R > \frac{2K}{1-\gamma}$.
\end{description}
\end{assumption}
Here, (i) is a Lyapunov-type condition that controls the growth of iterations with a Lyapunov function $w$, while (ii) is a Doeblin-type condition (see e.g., \cite{levin2009markov}) that assumes a support $\mu$ uniformly on the bounded level-set $\mathsf B$.

Under this assumption, a direct extension of Theorem 3.1 in \cite{hairer2011yet} shows that there exist constants $\bar \alpha \in(0,1)$ and $\beta > 0$, both depending on $\gamma, K$ and $\alpha$, such that 
\begin{align}
\lVert P^\pi [v] \rVert_{s,1+ \beta w} \leq \bar \alpha \lVert v \rVert_{s,1+\beta w}, \forall v \in \mathscr B_{1+\beta w}, \pi \in \Delta. \label{eq:mc:contraction}
\end{align}
This ensures the ergodicity of the underlying MCPs (cf.\ \cite[Chapter 10]{hernandez1999further}) and also guarantees the geometric convergence of value iterations. In Section \ref{sec:obj}, we shall generalize this set of conditions to nonlinear operators. 

\subsection{Risk maps}

\subsubsection*{Definition}\label{sec:riskmap} Inspired by risk measures applied in mathematical finance \cite{artzner1999coherent,follmer2002convex}, we introduce in the following our version of risk measures on the weighted norm space $\mathscr B_w$ and the probability measure space $\mathscr P$.

\begin{definition}\label{def:risk:measure}
 A mapping $\nu: \mathscr B_w \times \mathscr P \rightarrow \bar{\mathbb R}$ is said to be a \emph{risk measure} if it satisfies that for each $\mu \in \mathscr P$, 
 \begin{description}
  \item[\rm\it(i)] (monotonicity) $\nu(v, \mu) \leq \nu(u, \mu)$, whenever $v\leq u \in \mathscr B_w$;
  \item[\rm\it(ii)] (translation invariance) $\nu(v + u, \mu) = \nu(v, \mu) + u$, $\forall u \in \mathbb R, v \in  \mathscr B_w$; 
  \item[\rm\it(iii)] (centralization) $\nu(0, \mu) = 0$.
 \end{description}
 A mapping $\nu: \mathscr B_w \rightarrow \bar{\mathbb R}$ is said to be a \emph{risk measure with respect to (w.r.t.) $\mu \in \mathscr P$} if there exists a risk measure $\tilde \nu$ such that $\nu(v) = \tilde \nu(v, \mu)$, $\forall v \in \mathscr B_w$. 
 Furthermore, $\nu$ is said to be real-valued if $ \lvert \nu(v) \rvert < \infty$, $\forall v \in \mathscr B_w$.
\end{definition}

\begin{remark}\rm 
 Comparing with the definition of risk measures in \cite{artzner1999coherent,follmer2002convex}, we make the following extensions: 
 \begin{enumerate}
  \item We consider here the weighted space $\mathscr B_w$, adapted to a given Lyapunov function $w$,  rather than the space of bounded random variables, $L^\infty(\mu)$, since $\mathscr B_w$ is more suitable for investigating the stability properties of the underlying Markov process (see, e.g., \cite{meyn1993markov, hairer2011yet, Shen2013}) and is also more general than $L^\infty(\mu)$. We will show in later sections how to specify $w$, depending on the form of risk measures and the properties of the underlying Markov process as well.
  \item The objective probability measure $\mu$ is allowed to vary, since in the framework of MCPs, $Q_{x,a}(\cdot) := Q(\cdot|x,a)$ is a probability measure depending on different state-action pairs, though the transition kernel $Q$ itself is assumed to be fixed and known a priori.
  \item We dropped the assumption of coherency or convexity (see Definition \ref{def:convex} below for a formal definition), which is necessary for inducing risk-averse behavior. However, studies in behavioral economics (see e.g., prospect theory by \citet{tversky1992advances}) show that human agents are not always risk-averse. Hence, we drop this assumption and allow more general types of risk measures.
 \end{enumerate}
\end{remark}


To apply risk measures to MCPs, we follow the approach pioneered by \citet{ruszczynski2010risk} with, however, more general types of risk measures. 
We first introduce the concept of risk maps. A similar concept has also been introduced by \citet{Cavus2014} for coherency maps, and by \citet{Shen2013} for general risk maps. 
\begin{definition}
 Let $\{\mathsf X, \mathsf A, \{\mathsf A(x) | x \in \mathsf X \}, Q, c\}$ be an MCP and $\mathsf K = \{(x,a) | x \in \mathsf X, a \in \mathsf A(x)\}$ be the set of all feasible state-action pairs. A mapping $\mathcal R(v|x,a) : \mathscr B_w \times \mathsf K \rightarrow \mathbb R$ is said to be a \emph{risk map} on the MCP if 
\begin{description}
 \item[\it (i)] for each $(x,a) \in \mathsf K$, $\mathcal R(\cdot|x,a): \mathscr B_w \rightarrow \mathbb R$ is a real-valued risk measure w.r.t.\ $Q_{x,a}$, i.e., there exists a real-valued risk measure $\nu$ such that $\mathcal R_{x,a}(\cdot) = \nu(\cdot, Q_{x,a})$;
\item[\it(ii)] for each $v \in \mathscr B_w$, $\mathcal R(v|\cdot)$ is a real-valued $\mathcal B(\mathsf K)$-measurable function. 
\end{description}
 Furthermore, we define for any $\pi \in \Delta$, $\mathcal R^\pi(v|x) := \int_{\mathsf A(x)} \pi(\emph{d} a | x)\mathcal R(v|x,a).$
\end{definition}

For convenience, we write interchangeably $\mathcal R_{x,a}(v) := \mathcal R(v|x,a)$ and $\mathcal R_{x}^\pi(v) := \mathcal R^\pi(v|x)$. $\mathcal R^\pi(v)$ can therefore be viewed as a $\mathcal B(\mathsf X)$-measurable function for each $v \in \mathscr B_w$. It is worth to mention that $\mathcal R$ in fact depends on the transition kernel $Q$ of the underlying MCP, though for brevity we omit it in $\mathcal R$, because in the framework of MCPs, $Q$ is assumed to be fixed and known a priori.


\subsubsection*{Risk preference} We follow the rule that \emph{diversification} should be preferred if the agent is risk-averse (see e.g., \cite[Chapter 4]{follmer2004stochastic}). More specifically, suppose that an agent faces two choices under the same probability measure $\mu$, one of which has a future cost function $v$ while the other $v'$. If the agent diversifies, i.e., spends a fraction $\alpha$ of the resources on the first and the remaining amount on the other choice, the future cost is given by $\alpha v + (1-\alpha) v'$. If the applied risk measure is convex,
$\rho(\alpha v+ (1-\alpha)v', \mu) \leq  \alpha \rho(v, \mu) + (1-\alpha)\rho(v', \mu), \forall \alpha \in [0,1], $
then the diversification should reduce the risk. Thus, the agent's behavior is expected to be \emph{risk-averse}. Conversely, if the applied valuation function is \emph{concave}, the induced risk-preference should be \emph{risk-seeking}. This categorization can be extended to risk maps as follows.
\begin{definition}\label{def:convex}
 A risk map $\mathcal R$ on an MCP is said to be 
 \begin{itemize}
  \item \emph{convex}, if  $$\mathcal R_{x,a}(\alpha v+ (1-\alpha)v') \leq \alpha \mathcal R_{x,a}(v) + (1-\alpha)\mathcal R_{x,a}(v'), \forall \alpha \in [0,1], (x,a) \in \mathsf K;$$
  \item \emph{homogeneous}, if $\mathcal R_{x,a}(\lambda v ) = \lambda \mathcal R_{x,a}(v), \forall \lambda \in \mathbb R_+, (x,a) \in \mathsf K;$
  \item \emph{coherent}, if it is convex and homogeneous.
 \end{itemize}
\end{definition}
Among others, it is easy to verify that coherent risk maps are \emph{subadditive}:
\begin{align}
 \mathcal R_{x,a}(v + v') \leq \mathcal R_{x,a}(v) + \mathcal R_{x,a}(v'), \forall (x,a) \in \mathsf K, v, v' \in \mathscr B_w. \label{eq:subadditive}
\end{align}

\subsubsection*{Examples}

We list below several important examples that have been widely used in the literature. First of all, it is easy to verify that the conditional expectation
\begin{align}
 \mathcal R_{x,a}(v) = Q_{x,a}[v] := \int_{\mathsf X} Q(\diff y|x,a) v(y) \label{eq:mcps}
\end{align}
is a coherent risk map that is linear to $v$. In this case, we call the MCP \emph{risk-neutral}. 

\begin{example}\rm\label{ex:entropic}
\citet{follmer2002convex} introduced the \emph{entropic map}:
\begin{align}
 \mathcal R_{x,a}(v) := \dfrac{1}{\lambda}\ln \left( Q_{x,a} [e^{\lambda v}] \right) = \dfrac{1}{\lambda} \ln \left\lbrace \int_{\mathsf X} Q(\diff y|x,a) e^{ \lambda v(y)} \right\rbrace \label{eq:entropic}
\end{align}
where the risk-sensitive parameter $\lambda \in \mathbb R$ controls the risk-preference of $\mathcal R$: if $\lambda > 0$, $\mathcal R$ is everywhere convex and therefore everywhere risk-averse; if $\lambda < 0$, $\mathcal R$ is everywhere concave and therefore everywhere risk-seeking. It has been studied intensely also in the field of optimal control \cite{chung1987discounted,
  avila1998controlled, borkar2002risk, di2008infinite,coraluppi2000mixed,fleming1997risk,hernandez1996risk,
  marcus1997risk, jaskiewicz2007, cavazos2010optimality}. 
\end{example}

\begin{example}\label{ex:robust}
 \rm \citet{iyengar2005robust} introduced the framework of \emph{robust dynamic programming} (see also \cite{nilim2005robust}), by which he argues that in some applications the transition model $Q$ cannot be inferred exactly. Instead, he employs a set of transition probabilities, $\mathcal P$, which contains all possible ``ambiguous'' transition kernels. In order to gain the ``robustness'', the worst cost is considered, adapted to our framework, 
\begin{align}
\mathcal R_{x,a}(v) := \sup_{P(\cdot|x,a)\in \mathcal P_{x,a}} P_{x,a}[v] = \sup_{P(\cdot|x,a) \in \mathcal P_{x,a}} \int_{y \in \mathsf X} P(\diff y|x,a) v(y). \label{eq:rob}
\end{align}

One example of such $\mathcal P_{x,a}$ is given by \citet[Example 4.3]{ruszczynski2006optimization}:
\begin{align}
 \mathcal P_{x,a} := \left\{ \mu \in \mathscr P \middle| \mu \ll Q_{x,a},  0 \leq \gamma_1 \leq \frac{d \mu}{d Q_{x,a}} \leq \gamma_2 < \infty  \right\}.\label{eq:robust}
\end{align}
Here, ``$\ll$'' denotes the absolute continuity and $Q$ denotes the true transition kernel of the underlying MCP, while $\gamma_1$ and $\gamma_2$ control the degree of variation between the estimated transition kernel and its true model. In particular, if $\gamma_1 = 0$, it becomes the \emph{average value at risk} (see \cite{rockafellar2000Optconval,schied2009robust} and references therein). It is also notable that each concave and homogeneous valuation function has one dual representation of the form \eqref{eq:rob} under some regularity conditions for the set $\mathcal P$ (see e.g.\ \cite{Delbaen_2000} for essentially bounded spaces and \cite{svindland2009subgradients} for unbounded ones).
\end{example}

\begin{example}\label{ex:meanvar}
 \rm \citet{ogryczak1999stochastic} considered the trade-off between the one-step conditional mean and semideviation rather than the deviation of the whole Markov chain \cite{sobel1982variance, filar1989variance}:
\begin{align}
 \mathcal R_{x,a}(v) := Q_{x,a}[v] + \lambda \left( Q_{x,a} \left[ \left( v - Q_{x,a}[v] \right)^r_+\right] \right)^{1/r} \label{eq:meansemi}
\end{align}
where $r \geq 1$ and $\lambda \in [-1, 1]$ denotes the risk-preference parameter which controls the risk preference of $\mathcal R$: if $\lambda >0$, $\mathcal R$ is risk-averse; if $\lambda < 0$, $\mathcal R$ is risk-seeking. Here, $(x)_+ := \max(x,0)$. Setting $r=2$, this map can be viewed as an approximation of the mean-variance tradeoff scheme defined in \cite{filar1989variance}.
\end{example}

\begin{example}
 \rm The entropic map introduced above belongs in fact to a large family of risk maps called \emph{utility-based shortfall} (see \cite[Section 4.6]{follmer2004stochastic} and \cite[Section 2]{schied2009robust}) defined as follows. Let $u :\mathbb R \rightarrow \mathbb R$ be a continuous, increasing and non-constant utility function satisfying $u(0) = 0$. Assume that there exists a constant $m \in \mathbb R$ such that $\int_{\mathsf X} u(v(y) - m) Q_{x,a}(\diff y) < \infty$ for each $(x,a) \in \mathsf K$. Then, the utility-based shortfall is defined as 
\begin{align}
 \mathcal R_{x,a}(v) := \sup \left\{ m \in \mathbb R \mid \int_{\mathsf X} u(v(y) - m ) Q_{x,a}(\diff y) \geq 0 \right\}. \label{eq:shortfall}
\end{align}
It is remarkable that in principle, one can apply different $u$ for different state-action pair $(s,a)$. Here, for brevity, we apply the same $u$ for all $(s,a) \in \mathsf K$. 
\end{example}

\section{Average risk-sensitive MCPs}\label{sec:obj}
We define $\mathcal T: \mathscr B_w \times \mathsf K \rightarrow \mathbb R$ as
\begin{align*}
 \mathcal T(v|x,a) := c(x,a) + \mathcal R(v|x,a), (x,a) \in \mathsf K,
\end{align*}
and for each single-step policy $\pi \in \Delta$ define
\begin{align}
 \mathcal T^\pi(v|x) := \int \pi(\diff a |x) \mathcal T(v|x,a) = c^\pi(x) + \mathcal R^\pi(v|x) \label{eq:Tpi}
\end{align}
which is $\mathcal B(\mathsf X)$-measurable for each $v \in \mathscr B_w$. For convenience, we interchangeably write $\mathcal T_{x,a}(v) := \mathcal T(v|x,a)$ and $\mathcal T^\pi_x(v) := \mathcal T^\pi(v|x)$. It is worth to mention that for any $\pi \in \Delta$, monotonicity and translation invariance (but not centralization) (see Definition \ref{def:risk:measure}) of $\mathcal R$ imply 
the same properties for $\mathcal T^\pi_x(\cdot)$.


\subsection{Average criterion}\label{sec:ave:criterion}
Before we construct the risk-sensitive objective using the backward induction as in \cite{ruszczynski2010risk} and \cite{Shen2013}, we first state conditions, under which $\mathcal T^\pi(v|\cdot) \in \mathscr B_w$ for fixed $v \in \mathscr B_w$ and any $\pi \in \Delta$. For risk-neutral MCPs, where $\mathcal R(\cdot) = Q[\cdot]$, this is usually guaranteed by assuming $\lvert c(x,a) \rvert \leq C w(x)$ and $Q_{x,a}[w] \leq C w(x)$, for all $(x,a) \in \mathsf K$ with some positive constant $C$ (see e.g., \cite[Assumption 10.2.1(d) and (f)]{hernandez1999further}). For general risk maps which are not homogeneous, the above assumption is not sufficient. Instead, we consider the following assumption throughout this paper:
\begin{description}
 \item[\rm\lya{1}]  There exist a $\mathcal B(\mathsf X)$-measurable function $w_0: \mathsf X \rightarrow [0,\infty)$, constants $\gamma_0 \in (0,1)$ and $K_0 > 0$ such that $$\mathcal T_{x,a}(w_0) \vee \left( -\mathcal T_{x,a}(-w_0) \right) \leq \gamma_0 w_0(x) + K_0, \forall (x,a) \in \mathsf K.$$ 
\end{description}
Here, $a \vee b := \max(a,b)$. From now on, let $w := 1 + K^{-1} w_0$ with some positive $K \in \mathbb R_+$ whose value will be specified in Section \ref{sec:bounded}.

\begin{proposition}\label{prop:forward:weight}
Assume that \lya{1} holds and $v \in \mathscr B_{w}$ satisfies $\lvert v(x) \rvert \leq w_0(x)  + A, \forall x \in \mathsf X$, with some positive $A$. Then $\lvert \mathcal T_{x}^\pi(v) \rvert \leq w_0(x)  + A + K_0, \forall x \in \mathsf X$, $\pi \in \Delta$.
\end{proposition}
\begin{proof}
 By assumption, $-w_0 - A \leq v \leq w_0 + A$. \lya{1} implies that $$\mathcal T_{x,a}(v) \leq \mathcal T_{x,a}(w_0) +A \leq \gamma_0 w_0(x) + K_0 + A \leq w_0(x) + K_0 + A.$$
 Similarly, we obtain $\mathcal T_{x,a}(v) \geq - w_0(x) - K_0 - A$, by which the assertion follows. \quad
\end{proof}

Let $\bs \pi = [\pi_0, \pi_1, \ldots]$ be a Markov policy. The above proposition shows that starting from any $v \in \mathscr B_{w} $ satisfying $\lvert v \rvert \leq w_0  + A$, the following backward iteration
\begin{align*}
 v^{\bs \pi, T, T} := \mathcal T^{\pi_T}(v), v^{\bs \pi, t,T} := \mathcal T^{\pi_t} \left( v^{\bs \pi, t+1, T} \right), t = T-1, T-2, \ldots, 0
\end{align*}
is well-defined, since $v^{\bs \pi, t,T} \in \mathscr B_{w}$ for each $t = 0, 1, \ldots, T$. 

We now apply risk maps to construct risk-sensitive objectives as in \cite{ruszczynski2010risk, Shen2013}. Replacing the conditional expectation in \eqref{eq:decompose} with a risk map $\mathcal R$, we obtain for $\boldsymbol \pi \in \Pi_M$ and $x \in \mathsf{X}$,
\begin{align}
J_T(x,\boldsymbol \pi) := & c^{\pi_0}(x) + \mathcal R^{\pi_0}_{x} \left( c^{\pi_1} + \mathcal R^{\pi_1} \left(c^{\pi_2} + \ldots + \mathcal R^{\pi_{T-1}}(c^{\pi_T}) \ldots \right) \right) \label{eq:T}\\
       = & \mathcal T^{\pi_0}_x \left( \mathcal T^{\pi_1}(\cdots \mathcal T^{\pi_T}(0) \cdots) \right),  \nonumber
\end{align}
and the risk-sensitive objective considered in this paper is the \emph{average risk}:
\begin{align}
 J(x, \boldsymbol \pi) := &\limsup_{T \rightarrow \infty} \dfrac{1}{T}J_T(x, \boldsymbol \pi). \label{eq:avrk}
\end{align}

\begin{remark}\rm
 Applying the same recursive approach as above, other two widely used objectives in literature, the finite-stage total cost and the discounted cost, can be analogously extended to risk-sensitive objectives, the \emph{finite-stage total risk} and the \emph{discounted risk}, respectively (see \cite{ruszczynski2010risk} for coherent and \cite{Shen2013} for general risk maps). 
\end{remark}

\begin{remark}\rm
It is remarkable that the entropic map introduced in Exampe \ref{ex:entropic} is \emph{time-consistent} \cite{Kupper2009}. To see this in the framework of MCPs, note that for the finite-stage sum,
  \begin{align*}
    \nu(S_T|x) = \dfrac{1}{\lambda}\ln \left( \mathbb E^{\bs \pi} [e^{\lambda \sum_{t=0}^T c(X_t, A_t)} | X_0 = x] \right) = \dfrac{1}{\lambda}\ln \left( \mathbb E^{\bs \pi} [\prod_{t=0}^T e^{\lambda c(X_t, A_t)} | X_0 = x] \right).
  \end{align*}
  Suppose now that the policy is deterministic $\bs \pi = [f_0, f_1, \ldots, f_T]$. By the Markov property, we have then
  \begin{align*}
   \nu(S_T|x) = & \dfrac{1}{\lambda}\ln \left( e^{\lambda c^{f_0}(x)} P_x^{f_0}\left[  \cdots \left[ e^{\lambda c^{f_T}}P^{f_{T-2}}\left[e^{\lambda c^{f_T}} P^{f_{T-1}}[e^{\lambda c^{f_T}}] \right] \right] \cdots \right] \right) \\
    = & c^{f_0}(x) + \mathcal R_{x}^{f_0} \left( \cdots \left( c^{f_{T-1}} + \mathcal R^{f_{T-1}} \left( c^{f_{T}} \right) \right) \cdots \right).
  \end{align*}
However, this equivalence between the risk measure and its time-consistent version applied in this paper is an exception rather than the rule (for details see e.g., \cite{roorda2007time}). It does not hold for the other risk maps introduced in the last section. This also explains why we apply a recursive approach as in \eqref{eq:T} to enforce the time consistency instead of applying directly risk measures to the sum. 
\end{remark}

\subsubsection*{Remarks on Assumption \lya{1}} 
The Lyapunov-type condition \lya{1} implies a growth condition with the nonnegative weight function $w_0$. Sufficient conditions for \lya{1} in terms of the risk map $\mathcal R$ and the cost function $c$ separately are given as follows:
\begin{description}
\item[\rm\lya{1a}] $w_0: \mathsf X \rightarrow [0,\infty)$ is a $\mathcal B(\mathsf X)$-measurable function satisfying $$\mathcal R_{x,a}(w_0) \vee \left( - \mathcal R_{x,a}( - w_0) \right) \leq \hat \gamma_0 w_0(x) + \hat K_0, \forall (x,a) \in \mathsf K, \textrm{  and}$$
 \item[\rm\lya{1b}] $\lvert c(x,a) \rvert \leq C_0 (w_0^p(x) +1), \forall (x,a) \in \mathsf K, $ with some positive constants $p \in (0,1)$ and $C_0 > 0.$
 \end{description}
If, in addition, the risk map is convex, inducing risk-averse behavior, we have that $-\mathcal R_{x,a}(-w_0) \leq \mathcal R_{x,a}(w_0)$ holds for all $(x,a) \in \mathsf K$. Hence, \lya{1a} reduces to 
\begin{align*}
 \mathcal R_{x,a}(w_0) \leq \hat \gamma_0 w_0(x) + \hat K_0.
\end{align*}
Note that for any $\gamma \in (0, 1 - \hat \gamma_0)$, there exists a positive constant $\tilde K_0$ such that $C_0 (w_0^p + 1) \leq \gamma w_0 + K$. Hence, \lya{1b} yields $\lvert c(x,a) \rvert \leq \gamma w_0 + K$, which, together with \lya{1a} implies \lya{1}. 

Furthermore, we show below that for coherent risk maps, Assumption \lya{1} is in fact compatible with the conventional assumptions applied in the literature of MCPs (see e.g., \cite[Assumption 10.2.1(d) and (f)]{hernandez1999further} for risk-neutral MCPs).

\begin{assumption}\label{assp:weight}
 \rm There exist a $\mathcal B(\mathsf X)$-measurable function $\tilde w_0: \mathsf X \rightarrow [0,\infty)$, positive constants $C$, $\tilde \gamma_0 \in (0,1)$ and $\tilde K_0$ such that 
 \begin{description}
 \item[\rm\lya{1a'}] $ \mathcal R_{x,a}(\tilde w_0) \leq \tilde \gamma_0 \tilde w_0(x) + \tilde K_0, \forall (x,a) \in \mathsf K$, and
  \item[\rm\lya{1b'}] $\lvert c(x,a) \rvert \leq C (\tilde w_0(x) + 1), \forall (x,a) \in \mathsf K$.
 \end{description}
\end{assumption}
\begin{proposition} \label{prop:coherence:bound}
 If $\mathcal R$ is a coherent risk map, Assumption \ref{assp:weight} implies \lya{1} with  $w_0 = \frac{C}{\gamma_0 - \tilde \gamma_0} \tilde w_0$ and $K_0 = C + \frac{C \tilde K_0}{\gamma_0 - \tilde \gamma_0}$, for any $\gamma_0 \in (\tilde \gamma_0, 1)$. 
\end{proposition}
\begin{proof}
 Fix $\gamma_0 \in (\tilde \gamma_0, 1)$. \lya{1b'} and \lya{1a'} imply for each $(x,a) \in \mathsf K$, 
 \begin{align*}
  \mathcal T_{x,a}(w_0) = & c(x,a) + \mathcal R_{x,a}(w_0) \\
         \leq & C (1+\tilde w_0(x)) + \mathcal R_{x,a} (\frac{C}{\gamma_0 - \tilde \gamma_0} \tilde w_0) = C (1+\tilde w_0(x)) + \frac{C}{\gamma_0 - \tilde \gamma_0}  \mathcal R_{x,a}(\tilde w_0) \\
                        \leq & C ( 1 + \frac{\tilde \gamma_0}{\gamma_0 - \tilde \gamma_0})\tilde w_0(x) + K_0 = \gamma_0 w_0(x) + K_0.
 \end{align*}
Similarly, we obtain $-\mathcal T_{x,a}(- w_0) \leq \gamma_0 w_0(x) + K_0$ which implies \lya{1}. \quad
\end{proof}


\subsection{Bounded forward invariant subset}\label{sec:bounded}
Let $\boldsymbol \pi = [\pi_0, \pi_1, \ldots]$ be a Markov policy and define the following iteration
\begin{align*}
 \mathcal T^{(\bs \pi,n)}(v) := \mathcal T^{\pi_0} \left( \mathcal T^{\pi_1}(\cdots \mathcal T^{\pi_n}(v) \cdots) \right), n = 0, 1, \ldots.
\end{align*}
Hence, by definition, the $n$-stage risk-sensitive objective $J_n(x, \bs \pi) = \mathcal T^{(\bs \pi,n)}_x(0)$. 

Suppose Assumption \lya{1} holds and $v \in \mathscr B_w$ satisfies $\lvert  v\rvert \leq w_0 + A$ with some $A > 0$, then Proposition \ref{prop:forward:weight} shows that $\lvert \mathcal T^{(\bs \pi,n)}(v) \rvert \leq w_0 + K(n)$ with some positive $K(n) > 0$. Hence, $\mathcal T^{(\bs \pi,n)}(v) \in \mathscr B_w, \forall n\in \mathbb N$. Note that, however, $\{ K(n) \}$ may be an increasing sequence such that $K(n) \rightarrow \infty$ as $n \rightarrow \infty$. This implies that the sequence $\{ \mathcal T^{(\bs \pi,n)}(v), n=1, 2, \ldots\}$ can be unbounded w.r.t.\ $w$-norm. We will specify below a sufficient condition, similar to the ergodicity condition for risk-neutral MCPs (see Assumption \ref{assp:standard}), which implies boundedness of the sequence $\{ \mathcal T^{(\bs \pi,n)}(v)\}$ w.r.t.\ the $w$-seminorm.


\begin{assumption}\label{asmp:subset}\rm
Suppose that \lya{1} holds with some $\mathcal B(\mathsf X)$-measurable function $w_0: \mathsf X \rightarrow [0,\infty)$, constants $\gamma_0 \in (0,1)$ and $K_0 > 0$. In addition, assume
 \begin{description}
  \item[\rm\lya{2}] there exists a constant $K > K_0$ such that the inequality
\begin{align*}
 \mathcal R_{x,a}(w_0 + K) - \mathcal R_{x,a}(v) + \mathcal R_{y,b}(v)  -  \mathcal R_{y,b}(-w_0-K) \geq 2 K_0 
\end{align*}
holds for all $v\in \mathscr B_{1+w_0}$ satisfying $\lvert v\rvert \leq w_0 + K$, and $x, y \in \mathsf B_0 := \{x \in \mathsf X| w_0(x) \leq \frac{2K_0}{1-\gamma_0}\}$, $a \in \mathsf A(x)$, $b \in \mathsf A(y)$.
 \end{description}
\end{assumption}
We first state the main result, followed by several remarks on \lya{2}. 
\begin{theorem}\label{th:inv_bound}
 Suppose Assumption \ref{asmp:subset} holds. Then for each $\pi \in \Delta$, 
 \begin{align*}
  \lVert \mathcal T^\pi(v) \rVert_{s, 1+ K^{-1} w_0} \leq K, \textrm{whenever } \lVert v \rVert_{s, 1+ K^{-1} w_0} \leq K.
 \end{align*}
\end{theorem}
\begin{proof}
 Note that adding a constant to $v$ will not change the required inequality. Due to Lemma \ref{lm:semi}, we may assume that $\lvert v \rvert \leq K + w_0$. By the definition of the weighted seminorm, it is sufficient to show that for each $\pi \in \Delta$,
 \begin{align}
  \mathcal T_x^\pi(v) - \mathcal T_y^\pi(v) \leq 2 K + w_0(x) + w_0(y), \forall x \neq y \in \mathsf X.
 \end{align}

We consider the following two cases. Case I: $w_0(x) + w_0(y) \geq \frac{2K_0}{1-\gamma_0}$. This implies
\begin{align}
 2 K + \gamma_0 (w_0(x) + w_0(y)) + 2 K_0 \leq 2 K + w_0(x) + w_0(y). \label{eq:casei}
\end{align}
By the monotonicity of $\mathcal R$ and Assumption \lya{1}, we have
\begin{align*}
 \mathcal T_x^\pi(v) \leq \sup_{a \in \mathsf A(x)} \mathcal T_{x,a}(v) \leq & \sup_{a \in \mathsf A(x)} \mathcal T_{x,a}(K+w_0) = K + \sup_{a \in \mathsf A(x)} \mathcal T_{x,a}(w_0) \\
  \leq & K + \gamma_0 w_0(x) + K_0, \forall x \in \mathsf X, 
\end{align*}
and similarly $\mathcal T_y^\pi(v) \geq -K - \gamma_0 w_0(y) - K_0, \forall y \in \mathsf X.$
Hence, together with \eqref{eq:casei} we obtain
$
 \mathcal T_x^\pi(v) - \mathcal T_y^\pi(v) \leq 2 K + w_0(x) + w_0(y).
$

Case II: $w_0(x) + w_0(y) \leq \frac{2K_0}{1-\gamma_0}$. Then both $x$ and $y$ are in the subset $\mathsf B_0$. Hence, 
\begin{align*}
 \mathcal T_x^\pi(v) - \mathcal T_y^\pi(v) \leq & c^\pi(x) - c^\pi(y) + \mathcal R_{x}^\pi(v) - \mathcal R_{y}^\pi(v)\\
\textrm{(by \lya{2})} \ \leq & 2(K - K_0) + c^\pi(x) - c^\pi(y) + \mathcal R_{x}^\pi(w_0) - \mathcal R_{y}^\pi(w_0) \\
\textrm{(by \lya{1})} \ \leq & 2(K - K_0) + 2 K_0 + \gamma_0 w_0(x) + \gamma_0 w_0(y) \\
\leq & 2 K + w_0(x) + w_0(y).
\end{align*}
Combining I and II, we obtain the required inequality. \quad 
\end{proof}

The above theorem implies immediately the boundedness of iterations $\{ \mathcal T^{(\bs \pi,n)}(v) \}$ under the weighted seminorm with $w = 1+ K^{-1} w_0$. 
\begin{corollary}
 Suppose Assumption \ref{asmp:subset} holds. Then for each $\bs \pi \in \Pi_M$ and $n \in \mathbb N$,  $\lVert \mathcal T^{(\bs \pi,n)}(v) \rVert_{s, 1+ K^{-1} w_0} \leq K, \textrm{whenever } \lVert v \rVert_{s, 1+ K^{-1} w_0} \leq K.$
\end{corollary}

Hence, given a risk map satisfying \lya{1} and \lya{2}, we can restrict ourselves to the \emph{bounded forward invariant subset} $\mathscr B_w^{(K)}$ (with $w = 1 + K^{-1} w_0 $) defined as
\begin{align*}
 \mathscr B_w^{(K)} := \left\{ v \in \mathscr B_w \mid \lVert v \rVert_{s,w} \leq K \right\},
\end{align*}
rather than the whole set $\mathscr B_w$.

\subsubsection*{Remarks on \lya{2}}
\lya{2} in fact generalizes Doeblin condition stated in Assumption \ref{assp:standard}(ii). To see this, taking the risk-neutral risk map $\mathcal R_{x,a}(\cdot) = Q_{x,a}[\cdot]$, \lya{2} becomes
$
 Q_{x,a}[w_0+K - v] -  Q_{y,b}[-v-w_0-K] \geq 2K_0 
$
due to the linearity of $Q$. Suppose now $Q$ satisfies Assumption \ref{assp:standard}(ii) on $\mathsf B_0$, i.e., there exists a probability measure $\mu$ and some constant $\alpha \in (0,1)$ such that $Q_{x,a}(\mathsf B) \geq \alpha \mu(\mathsf B), \forall \mathsf B \in \mathcal B(\mathsf X)$, $x \in \mathsf B_0$ and $a \in \mathsf A(x)$. Then, it is easy to see that
\begin{align*}
 Q_{x,a}[w_0+K - v] -  Q_{y,b}[-v-w_0-K] \geq 2  \alpha \mu[ w_0+K ] \geq 2 \alpha K.
\end{align*}
Hence, given $K_0$ and $\alpha$, we can easily choose $K$ such that \lya{2} holds. This connection holds also for coherent risk maps, which we will show below. For more general risk maps, however, the connection is not straightforward. We will show the sufficient conditions for \lya{2} in the case of entropic maps in Section \ref{sec:entropic} and some classes of utility-based shortfalls in Section \ref{sec:shortfall}. 

\subsubsection*{Conditions for coherent risk maps}
For a special case, when applying coherent risk maps, we show below that the above set of conditions can be replaced by a more conventional set of conditions. 
\begin{assumption}\label{asmp:coherent} \rm
Suppose \lya{1b'} and \lya{1a'} hold with some nonnegative $\mathcal B(\mathsf X)$-measurable function $\tilde w_0$, positive constants $C$, $\tilde \gamma_0 \in (0,1)$ and $\tilde K_0$. In addition,
\begin{description}
 \item[\rm \lya{2'}] there exist a probability measure $\mu \in \mathscr P_{1+\tilde w_0}$, a constant $\alpha > 0$ such that 
  \begin{align*}
   \mathcal R_{x,a}(v) - \mathcal R_{x,a}(u) \geq \alpha \mu[v - u], \forall v \geq u \in \mathscr B_{1 + \tilde w_0}
  \end{align*}
  holds for all $x \in \mathsf B_0 := \{ x \in \mathsf X | \tilde w_0(x) \leq 2 R_0 \}$ and $a \in \mathsf A(x)$, with some $R_0 > \tilde K_0/(1-\tilde \gamma_0)$.
\end{description}
\end{assumption}

\begin{proposition}\label{prop:coherent}
Let $\mathcal R$ be a coherent risk map w.r.t.\ an MCP. Suppose Assumption \ref{asmp:coherent} holds. Then $\gamma_0 := \frac{R_0 - \tilde K_0 + \tilde \gamma_0}{1+R_0}\in (\tilde \gamma_0, 1)$ and Assumption \ref{asmp:subset} holds with $\gamma_0$, $w_0 = \frac{C}{\gamma_0 - \tilde \gamma_0}\tilde w_0, K_0 = C + \frac{C \tilde K_0}{\gamma_0 - \tilde \gamma_0} \textrm{ and } K = K_0 / \alpha.$
\end{proposition}
\begin{proof}
First, the inequalities $\tilde \gamma_0 < \gamma_0 < 1$ can be easily verified and therefore the proof is omitted here. Second, Proposition \ref{prop:coherence:bound} implies \lya{1}. 

It remains to verify \lya{2'}. First, it is easy to verify that $\frac{C R_0}{\gamma_0 - \tilde \gamma_0} = \frac{K_0}{ 1 - \gamma_0}$ and 
\begin{align*}
 \mathsf B_0 =  \left\{ x \in \mathsf X \middle|\tilde w_0(x) \leq 2 R_0 \right\} = & \left\{ x \in \mathsf X \middle| w_0(x) \leq \frac{2 K_0}{ 1 - \gamma_0} \right\}. 
\end{align*}
Let $w' := w_0 + K$. Then for any $v \in \mathscr B_w$ satisfying $\lvert v \rvert \leq w'$ and $x \in \mathsf B_0, a \in \mathsf A(x)$, \lya{2'} implies $\mathcal R_{x,a}(v) - \mathcal R_{x,a}(w') \leq \alpha \mu(v - w')$ and hence,
\begin{align*}
 \mathcal R_{x,a}(v) - \alpha \mu(v) \leq & \mathcal R_{x,a}(w') - \alpha \mu(w') \\
 = & (1-\alpha) K + \mathcal R_{x,a} (w_0) - \alpha \mu(w_0) \leq (1-\alpha) K + \mathcal R_{x,a} (w_0).
\end{align*}
Repeating the above derivation for $-v \leq w'$, we obtain $\mathcal R_{x,a}(-v) - \alpha \mu(-v) \leq (1-\alpha) K + \mathcal R_{x,a} (w_0)$. The coherency of $\mathcal R$ yields $\mathcal R_{x,a}(-v) \geq - \mathcal R_{x,a}(v)$ and hence, 
\begin{align*}
 - \mathcal R_{x,a}(v) + \alpha \mu(v) \leq (1-\alpha) K + \mathcal R_{x,a} (w_0).
\end{align*}
Thus, for any $x, y \in \mathsf B_0, a \in \mathsf A(x), b \in \mathsf A(y)$, we have
\begin{align*}
 \mathcal R_{x,a}(v) - \mathcal R_{y, b}(v) \leq 2 (1-\alpha) K + \mathcal R_{x,a}(w_0) + \mathcal R_{y,b}(w_0).
\end{align*}
This implies that \lya{2} holds with $K = K_0 / \alpha$. \quad
\end{proof}


\subsection{Geometric contraction} \label{sec:geo:cont}
In this subsection, we state sufficient conditions, under which $\mathcal R^\pi$ is a contraction, similar to the geometric contraction for standard Markov chains stated in \eqref{eq:mc:contraction}. We first introduce the concept of \emph{upper envelope}.
\begin{definition}
 A coherent risk map $\bar{\mathcal R}^{(w,K)}$ is said to be an \emph{upper envelope} of a valuation map $\mathcal R$ given a bound $K \in \mathbb R_+$, if for all $v, u \in \mathscr B_w^{(K)}$,
  $$
  \mathcal R_{x,a}(v) - \mathcal R_{x,a}(u) \leq \bar{\mathcal R}^{(w,K)}_{x,a}(v-u), \forall (x,a) \in \mathsf K.$$
\end{definition}

Let $w_0$ be a function satisfying Assumption \lya{1} and $w = 1+ K^{-1} w_0$ with $K$ satisfying \lya{2}. Analogous to the assumption applied to risk-neutral MCPs (see Assumption \ref{assp:standard}) and the one applied in risk-sensitive MCPs by \citet{Shen2013}, we introduce a set of conditions on the upper envelope.
\begin{assumption}\label{assp:upper:envelope}\rm
 \begin{description}
  \item[\rm\ue{1}] there exist constants $\gamma \in (0,1)$, $\bar K > 0$ and an upper envelope $\bar{\mathcal R}^{(w,K)}$ such that
  \begin{align*}
    \bar{\mathcal R}^{(w,K)}_{x,a}(w_0) \leq \gamma w_0(x) + \bar K, \forall (x,a) \in \mathsf K. 
  \end{align*}
  \item[\rm\ue{2}] there exist a constant $\alpha \in (0,1)$ and a probability measure $\mu\in \mathscr P_{1+w_0}$ satisfying
  \begin{align*}
   \bar{\mathcal R}^{(w,K)}_{x,a}(v) - \bar{\mathcal R}^{(w,K)}_{x,a}(u) \geq \alpha \mu[v-u], \forall x \in \mathsf B, a\in \mathsf A(x), v \geq u \in \mathscr B_{1+w_0}
  \end{align*}
  where $\mathsf B := \{x \in \mathsf X| w_0(x) \leq R \}$ for some $R > \frac{2 \bar K}{1 - \gamma}$. 
 \end{description}
\end{assumption}
\begin{remark}\label{rm:contraction}
 \rm Apparently, if $\mathcal R$ is coherent, then $\mathcal R$ is an upper envelope of itself for all bounds $K > 0$, due to its subadditivity stated in \eqref{eq:subadditive}. Hence, Assumption \ref{assp:upper:envelope} is equivalent to Assumptions \lya{1a'} and \lya{2'}. On the other hand, we have shown in Proposition \ref{prop:coherent} that  Assumptions \lya{1b'}, \lya{1a'} and \lya{2'} imply Assumption \ref{asmp:subset}. Hence, Assumptions \ref{asmp:subset} and \ref{assp:upper:envelope} can be reduced to Assumption \ref{asmp:coherent}.
\end{remark}

We state now the main result of this subsection.
\begin{theorem}\label{th:pe:contraction}
 Suppose Assumption \ref{assp:upper:envelope} holds. Then there exists a constant $\bar \alpha \in (0,1)$ and $\beta > 0$ such that $$\lVert \mathcal R^\pi(v) - \mathcal R^\pi(u) \rVert_{s,1 + \beta w_0} \leq \bar \alpha \lVert v - u \rVert_{s,1 + \beta w_0}, \forall v, u \in \mathscr B_{w}^{(K)}, \pi \in \Delta.$$
\end{theorem}
\begin{proof}
 The proof is in essence similar to the proof of \citet[Theorem 3.11]{Shen2013}. For readers' convenience, we incorporate it in the Appendix.  
\end{proof}

The Lyapunov-type condition \ue{1} yields the following result. 
\begin{lemma}\label{lm:uni}
 Suppose that Assumptions \ref{asmp:subset} and \ue{1} hold. 
 Then $$\lim_{n \rightarrow \infty} \frac{1}{n} \lVert \mathcal T^{(\boldsymbol \pi,  n)}(v) - \mathcal T^{(\boldsymbol \pi,  n)}(u) \rVert_{w} = 0, \forall v,u \in \mathscr B_{w}^{(K)}, \bs \pi \in \Pi_M.$$
\end{lemma}
\begin{proof}
 It is sufficient to show that $\lVert \mathcal T^{(\boldsymbol \pi,  n)}(v) - \mathcal T^{(\boldsymbol \pi,  n)}(u) \rVert_{w}$ is uniformly bounded. Indeed, let $\mathcal U_x^\pi(v) := \int \pi(\diff a| x) \bar{\mathcal R}^{(w,K)}_{x,a}(v)$ and $K' := \frac{\bar K}{K} +1 - \gamma > 0$. By \ue{1}, we have for each $x \in \mathsf X$ and $\pi \in \Delta$, 
 \begin{align*}
  \mathcal T^\pi_x(v)- \mathcal T^\pi_x(u) \leq \mathcal U_x^\pi(\lvert v- u\rvert) \leq & \lVert v- u\rVert_{w} \left(\frac{\gamma}{K} w_0(x) + \frac{\bar K}{K} + 1 \right), 
 \end{align*}
 which implies that $\lvert \mathcal T^\pi_x(v)- \mathcal T^\pi_x(u) \rvert \leq \lVert v- u\rVert_{w} (\gamma w(x) + K' ).$

In addition, by Theorem \ref{th:inv_bound}, $\lVert \mathcal T^{(\boldsymbol \pi,  n)}(v) \rVert_{s,w} \leq K$ holds for all $n \in \mathbb N$. Hence, by induction w.r.t.\ $n$, we have for $n = 2, 3, \ldots$
\begin{align*}
 \lvert \mathcal T^{(\boldsymbol \pi,  n)}_x(v)- \mathcal T^{(\boldsymbol \pi,  n)}_x(u) \rvert \leq & \mathcal U_x^{\pi_0}(\lvert \mathcal T^{([\pi_1, \ldots], n-1)}(v) - \mathcal T^{([\pi_1, \ldots], n-1)}(u)\rvert) \\
 \leq & \lVert v - u \rVert_w \mathcal U_x^{\pi_0} \left( \gamma^{n-1} w + K' \sum_{k=0}^{n-2} \gamma^k \right) \\
 \leq & \lVert v - u \rVert_w \left( \gamma^{n} w(x) + K' \sum_{k=0}^{n-1} \gamma^k \right), \forall x \in \mathsf X, 
\end{align*}
which implies that $\lVert \mathcal T^{(\boldsymbol \pi,  n)}(v) - \mathcal T^{(\boldsymbol \pi,  n)}(u) \rVert_{w} \leq \frac{K'}{1-\gamma}$. \quad
\end{proof}

\subsection{Nonlinear Poisson equation}
Similar to the risk-neutral average MCPs (see, e.g., \cite[Chapter 10]{hernandez1999further}), the optimization problem of average risk-sensitive objective defined in \eqref{eq:avrk} is closely related to the following nonlinear Poisson equation
\begin{align}
  \rho + h(x) = \inf_{a \in \mathsf A(x)} \left( c(x,a) + \mathcal R(h|x,a) \right), \forall x \in \mathsf X, \label{eq:pe}
\end{align}
where $(\rho,h) \in \mathbb R \times \mathscr B_{w}$ is its solution. 
Define the following operator
\begin{align}
 \mathcal F_x(v) = \mathcal F(v|x) := \inf_{a \in \mathsf A(x)} \left( c(x,a) + \mathcal R(v|x,a) \right). \label{eq:F}
\end{align}

\subsubsection*{Existence of optimal selectors}
In order to guarantee the existence of an ``optimal selector'' and the measurablity of the operator $\mathcal F$, we assume
\begin{assumption}\label{assp:select}
\rm For all $x \in \mathsf X$,
 \begin{description}
  \item[\rm\os{1}] the action space $\mathsf A(x)$ is compact;
  \item[\rm\os{2}] the cost function $c(x,a)$ is lower semicontinuous (l.s.c.) on $\mathsf A(x)$;
  \item[\rm\os{3}] $a \mapsto \mathcal R(v|x,a)$ is l.s.c.\ on $\mathsf A(x)$, for each $v \in \mathscr B_w$ satisfying $\lvert v \rvert \leq K + w_0$.
 \end{description}
\end{assumption}

Note that each $v \in \mathscr B_w^{(K)}$ can be represented as $v = \tilde v + C$ with some constant $C$ and $\tilde  v \in \mathscr B_w$ satisfying $\lvert \tilde v \rvert \leq K + w_0$. Due to the translation invariance, \os{3} immediately imply that $a \mapsto \mathcal R(v|x,a)$ is l.s.c.\ on $\mathsf A(x)$, for each $v \in \mathscr B_w^{(K)}$ and $x \in \mathsf X$. Hence, a direct application of \cite[Lemma 8.3.8(a)]{hernandez1999further} yields
\begin{proposition}
 Let $\mathcal R$ be a risk map on an MCP. Suppose Assumption \ref{assp:select} holds. Then for each $v \in \mathscr B_w^{(K)}$ and $x \in \mathsf X$, there exists a deterministic policy $f \in \Delta_D$ such that
 \begin{align*}
  c^f(x) + \mathcal R^f(v|x) = \mathcal F(v|x) := \inf_{a \in \mathsf A(x)} \left( c(x,a) + \mathcal R(v|x,a) \right)
 \end{align*}
and furthermore, $\mathcal F(v|\cdot)$ is $\mathcal B(\mathsf X)$-measurable.
\end{proposition}

\begin{remark}\rm 
 In risk-neutral MCPs, where $\mathcal R_{x,a}(\cdot) = Q_{x,a}[\cdot]$, \os{3} is satisfied (see e.g., \cite[Lemma 8.3.7(a)]{hernandez1999further}) if 
 \begin{description}
  \item[\rm\os{3a}] $a \mapsto \mathcal R_{x,a}[v]$ is l.s.c.\ on $\mathsf A(x)$, for each $v \in \mathscr B$; and
  \item[\rm\os{3b}] $a \mapsto \mathcal R_{x,a}[w_0]$ is l.s.c.\ on  $\mathsf A(x)$.
 \end{description}
For general risk maps, we need to further assume that $\mathcal R$ is l.s.c.\ w.r.t.\ $v$ on $\mathscr B_w^{(K)}$: 
\begin{description}
 \item[\rm\os{3c}] Let $\{ v_n \in \mathscr B, n=0,1,\ldots \}$ be a sequence of functions that pointwise converges to $v \in \mathscr B_w^{(K)}$. Then for all $(x,a) \in \mathsf K$, $\liminf_{n\rightarrow \infty}\mathcal R(v_n|x,a) \geq \mathcal R(v|x,a)$.
\end{description}
By extending the proof of \cite[Lemma 8.3.7(a)]{hernandez1999further}, \os{3a} --  \os{3c} imply \os{3}. Note that in the case of risk-neutral MCPs, \os{3c} is already satisfied due to Fatou's lemma. However, it does not hold for general risk maps (see \cite{Delbaen_2000} for a detailed discussion). Note that because verification of \os{3c} is usually not straightforward, we will verify \os{3} directly for various examples of risk maps shown in Section \ref{sec:examples}. 

\end{remark}

\subsubsection*{Existence of a unique solution} We show now the operator $\mathcal F$ is also a contraction under the weighted seminorm. 
\begin{lemma}\label{lm:F}
 Suppose Assumptions \ref{assp:upper:envelope} and \ref{assp:select} hold. Then there exists a constant $\bar \alpha \in (0,1)$ and $\beta > 0$ such that $$\lVert \mathcal F(v) - \mathcal F(u) \rVert_{s,1 + \beta w_0} \leq \bar \alpha \lVert v - u \rVert_{s,1 + \beta w_0}, \forall v, u \in \mathscr B_{w}^{(K)}.$$
\end{lemma}
\begin{proof}
 By definition we have
 \begin{align*}
  \mathcal F_x(v) - \mathcal F_x(u) \leq & \sup_{a \in \mathsf{A(x)}} \left\{ \mathcal R_{x,a}(v) - \mathcal R_{x,a}(u) \right\} \leq \sup_{a \in \mathsf{A(x)}} \bar{\mathcal R}^{(w,K)}_{x,a}(v-u), \forall x \in \mathsf X.
 \end{align*}
Define $\mathcal U_x(v) := \sup_{a \in \mathsf{A(x)}} \bar{\mathcal R}^{(w,K)}_{x,a}(v)$ and we have
\begin{align*}
 \mathcal F_x(v) - \mathcal F_x(u) \leq \mathcal U_x(v - u) \leq \mathcal U_x(\lvert v - u \rvert).
\end{align*}
The rest of the proof is similar to the proof of Theorem \ref{th:pe:contraction}. \quad
\end{proof}

\begin{theorem}\label{th:pe}
 Suppose that Assumptions \ref{asmp:subset}, \ref{assp:upper:envelope} and \ref{assp:select} hold. Then  there exist a unique $\rho\in \mathbb R$ and $ h \in \mathscr B_{w}$ such that $(\rho, h)$ satisfies the  Poisson equation
 \begin{align}
  \rho + h(x) = \inf_{a \in \mathsf A(x)} \left( c(x,a) + \mathcal R(h|x,a) \right), \forall  x \in \mathsf X.
 \end{align}
 In particular, if $\mathcal R$ is coherent, the assertion holds under Assumptions \ref{asmp:coherent} and \ref{assp:select}.
\end{theorem}
\begin{proof}
Let $\hat w := 1+\beta w_0$ as in Lemma \ref{lm:F}. Then the map $\mathcal F: \mathscr B_w^{(K)} \rightarrow \mathscr B_w^{(K)}$ defined in \ref{eq:F} is a contraction under $\hat w$-seminorm. Note that $\mathscr B_w^{(K)} \subset \mathscr B_{\hat w}$. We now extend the fixed-point theorem w.r.t.\ span-seminorm (cf.\ p.\ 321 \cite{arapostathis1993discrete} for bounded $\hat w$) to $\hat w$-seminorm. Let $\widetilde{\mathscr B}_{\hat w} = \mathscr B_{\hat w} / \sim$ be the quotient space, which is induced by the equivalence relation $\sim$ on $\mathscr B_{\hat w}$ defined by $v \sim u$ if and only if there exists some constant $C \in \mathbb R$ such that $v(x) - u(x) = C$ for all $x \in \mathsf X$, endowed with the quotient norm induced by the $\hat w$-seminorm. For $v \in \mathscr B_{\hat w}$, let $\tilde v$ be the corresponding equivalence class in $\widetilde{\mathscr B}_{\hat w}$ and $\widetilde{\mathcal F}: \widetilde{\mathscr B}_{\hat w} \rightarrow \widetilde{\mathscr B}_{\hat w}$ be the canonically induced map, i.e., $\widetilde{\mathcal F}(\tilde v) := \widetilde{\mathcal F(v)}$, $v \in \mathscr B_{\hat w}$. Since $\mathcal F$ is a contraction w.r.t.\  $\hat w$-seminorm on $\mathscr B_w^{(K)} \subset \mathscr B_{\hat w}$, $\widetilde{\mathcal F}$ is a contraction on $\{ v \in \widetilde{\mathscr B}_{\hat w} \mid \lVert v \rVert_{s,w} \leq K\}$ and therefore has a unique fixed point. Conversely, it follows that the map $\mathcal F$ has a $\hat w$-seminorm fixed point. In other words, there exists $h \in \mathscr B_w^{(K)}$ such that $\lVert \mathcal F(h) - h \rVert_{s,w} = 0$ and $\rho := \mathcal F(h) - h$ is a constant.

Next we show that such $\rho$ is unique. Suppose there exits another solution $(\rho',h') \in \mathbb R \times \mathscr B_w^{(K)}$ to the Poisson equation. Then $\mathcal F^n(h') = n \rho' + h'$ and $\mathcal F^n(h) = n \rho + h$. However, by Lemma \ref{lm:uni}, we have
$$\frac{1}{n} \lVert  n \rho' + h' - n \rho - h \rVert_w =\frac{1}{n}\lVert \mathcal F^n(h') - \mathcal F^n(h) \rVert_w \rightarrow 0, $$  
which implies $\rho' = \rho$. \quad
\end{proof}

\subsection{Solution to average risk-sensitive MCPs}
We show below that the unique real value $\rho$ satisfying the Poisson equation is in fact \emph{optimal} for the average risk-sensitive objective defined in \eqref{eq:avrk}.
\begin{theorem}
Suppose that Assumptions \ref{asmp:subset}, \ref{assp:upper:envelope} and \ref{assp:select} hold. Let $(\rho, h)$ be a solution to the Poisson equation defined in \eqref{eq:pe}. Then $\rho = J^*(x) = J(x,f^\infty), \forall x \in \mathsf X$, where $f$ denotes the optimal selector in the right-hand side of the Poisson equation. In particular, if $\mathcal R$ is coherent, the assertion holds under Assumptions \ref{asmp:coherent} and \ref{assp:select}.
\end{theorem}
\begin{proof}
 Let $\mathcal T^\pi$ be the operator defined in \eqref{eq:Tpi}. By its definition, $\mathcal F(h) = \mathcal T^f(h)$ and the induction
$
 (\mathcal T^f)^n (h) =  (\mathcal T^f)^{n-1} (\rho + h) = n \rho + h, n= 1,2,\ldots, 
$
yields $\dfrac{1}{n} \lVert (\mathcal T^f)^n (h) - \rho \rVert_w \rightarrow 0$ as $n \rightarrow \infty$. 
On the other hand, for any $v \in \mathscr B_w^{(K)}$, by Lemma \ref{lm:uni}, $\dfrac{1}{n}\lVert (\mathcal T^f)^n (v) - (\mathcal T^f)^n (h) \rVert_w \rightarrow 0$ implies that $$J(x,f^\infty) = \lim_{n \rightarrow \infty} \frac{1}{n} (\mathcal T^f)^n_x(0) = \rho, \forall x \in \mathsf X.$$

Next we prove that $\rho \leq J(x,\bs \pi)$ for all $\bs \pi \in \Pi_M = \Delta^\infty$ and $x \in \mathsf X$. In fact,  let $\bs \pi = [\pi_0,\pi_1, \ldots]$ be an arbitrary Markov policy. By definition $h \leq \mathcal T^\pi(h) - \rho, \forall \pi \in \Delta$. Iterating this inequality yields $h \leq \mathcal T^{\pi_0}( \mathcal T^{\pi_1}( \cdots \mathcal T^{\pi_{n-1}}(h)\cdots)) - n \rho,$ and hence
\begin{align}
 0 \leq \limsup_{n \rightarrow \infty} \frac{1}{n} \mathcal T^{\pi_0}(\mathcal T^{\pi_1} (\cdots\mathcal T^{\pi_{n-1}}(h)\cdots)) - \rho. \label{eq:rhoh}
\end{align}
Note that by definition $J(x, \bs \pi) = \limsup_{n \rightarrow \infty} \frac{1}{n} \mathcal T^{\pi_0}_x(\mathcal T^{\pi_1} (\cdots \mathcal T^{\pi_{n-1}}(0)\cdots))$. Lemma \ref{lm:uni} yields $J(x, \bs \pi) = \limsup_{n \rightarrow \infty} \frac{1}{n} \mathcal T^{\pi_0}_x(\mathcal T^{\pi_1} (\cdots\mathcal T^{\pi_{n-1}}(h)\cdots)).$ Hence, \eqref{eq:rhoh} implies 
$$\rho \leq \inf_{\bs \pi \in \Pi_M} J(\bs \pi) = J^*.$$
Since $f^\infty$ is a valid Markov policy in $\Pi_M$, $\rho = J^*(x) = J(x, f^\infty), \forall x \in \mathsf X$.
\quad 
\end{proof}

\paragraph*{Value iteration} We can use the following iterative procedure to calculate the optimal average risk: start from any $v_0 \in \mathscr B_w^{(K)}$ and iterate 
\begin{align*}
 v_{n+1} = \mathcal F(v_n) \textrm{ with } f_{n} \textrm{ being its selector }, n = 0, 1, 2, \ldots. 
\end{align*}
Lemma \ref{lm:F} and Theorem \ref{th:pe} ensure that $v_n \rightarrow \rho$ and $f_n \rightarrow f$, as $n \rightarrow \infty$, where $\rho = J^*$ is the unique solution to the Poisson equation and $f$ the optimal policy. In particular, this iteration will geometrically converge with rate $\bar\alpha$ as in Lemma \ref{lm:F}.

\section{Entropic maps}\label{sec:entropic}
In this section, we investigate the sufficient conditions particularly for the entropic map, which is the most widely used risk measure in the last four decades. Recall that in Example \ref{ex:entropic}, the entropic map is defined as
\begin{align*}
 \mathcal R_{x,a}(v) := \dfrac{1}{\lambda}\ln \left( Q_{x,a} [e^{\lambda v}] \right) = \dfrac{1}{\lambda} \ln \left\lbrace \int_{\mathsf X} Q(\diff y|x,a) e^{ \lambda v(y)} \right\rbrace.
\end{align*}
We consider in this section mainly the case $\lambda > 0$ and hence $\mathcal R$ is convex, which induces risk-averse behavior, though the results stated below can be easily extended to negative $\lambda$. From now on, without loss of generality, we set $\lambda = 1$.

\subsection{Upper envelope}
We first derive its upper envelope. 
\begin{proposition}\label{prop:ent}
 Let $\mu$ be a probability measure on $(\mathsf X, \mathcal B(\mathsf X))$ and $\nu(v) := \ln \left( \mu \left[e^v \right] \right)$. Suppose $\mu[e^{v}] < \infty$ holds for all $v \in \mathscr B_w$. Then 
\begin{itemize}
 \item[(i)]  $\nu(v) \leq \frac{\mu[e^v v]}{\mu[e^v]}$, and
 \item[(ii)] $\nu(v) - \nu(u) \leq \sup_{f \in \mathscr B_w^{(K)}} \frac{\mu[ e^{f} (v-u)]}{\mu[ e^{f} ]}$, $\forall v, u \in \mathscr B_w^{(K)}$.
\end{itemize}
\end{proposition}
\begin{proof}
 Given any two $u,v \in \mathscr B_w$, we obtain
 \begin{align}
  \nu(v) - \nu(u) = \ln \frac{\mu[e^{v}]}{\mu[e^u]} = \ln \frac{\mu[e^u e^{v-u}]}{\mu[e^u]} \geq \frac{\mu[e^u(v-u)]}{\mu[e^u]},  \label{eq:ent:jen}
 \end{align}
where the last inequality is due to Jensen's inequality. Hence, 
\begin{align*}
 \ln(\mu [e^v]) \geq \frac{\mu [e^u v]}{\mu [e^u]} - \mu\left[\frac{e^u}{\mu[e^u]}(u - \ln (\mu [e^u]))\right], \forall u, v \in \mathscr B_w.
\end{align*}
Restricting $u$ and $v$ to be in the subset $\mathscr B_w^{(K)}$, the above inequality yields
\begin{align*}
 \ln(\mu [e^v]) \geq \sup_{\xi = \frac{e^u}{\mu[e^u]}, u \in \mathscr B_w^{(K)}} \mu[\xi v] - \mu[\xi \ln(\xi)].
\end{align*}
Since the equality holds by taking $\xi^* := \frac{e^v}{\mu[e^v]}$, we obtain
\begin{align}
 \ln(\mu [e^v]) = \sup_{\xi = \frac{e^u}{\mu[e^u]}, u \in \mathscr B_w^{(K)}} \mu[\xi v] - \mu[\xi \ln(\xi)]. \label{eq:ent:dual}
\end{align}
The second term $\mu[\xi \ln(\xi)]$ on the right-hand side of the above equation is the \emph{relative entropy} and is always nonnegative (for proof see, e.g., \cite[Section 5.1]{ledoux2001concentration}). Hence, we obtain (i). Finally, (ii) follows from
\begin{align*}
 \ln(\mu [e^v]) - \ln(\mu [e^u]) \leq & \sup_{\xi = \frac{e^{f}}{\mu[e^{f}]}, f \in  \mathscr B_w^{(K)}} \mu[\xi (v - u)] = \sup_{f \in  \mathscr B_w^{(K)}} \frac{\mu[ e^{f} (v-u)] }{\mu [e^{f}]}. \quad
\end{align*}
\end{proof}

\begin{remark}\rm 
  The inequality in \eqref{eq:ent:dual} is similar to the dual representation of convex risk measures on $L^\infty$ \cite{follmer2002convex, follmer2004stochastic} or on more general spaces such as \emph{Orlicz hearts} \cite{cheridito2009risk}. However, since we consider a different function space, i.e., the weighted norm space, the existing result cannot be directly applied here. On the other hand, for other types of convex valuation functions, their dual representation provide us a generic approach to calculate their upper envelopes, as shown in the above proposition.
\end{remark}

By Proposition \ref{prop:ent}, we obtain one upper envelope for the entropic map:
\begin{align}
 \bar{\mathcal R}_{x,a}^{(w,K)}(u) = \sup_{f \in \mathscr B_w^{(K)}} \frac{Q_{x,a} [e^{f} u]}{Q_{x,a}[e^{f}]}, \label{eq:ent:upenv}
\end{align}
provided that $Q_{x,a} [e^{f}] < \infty$ holds for all $f \in \mathscr B_w$ and $(x,a) \in \mathsf K$. We show below how this condition is satisfied. 

\begin{assumption}\label{ass:lya} \rm
There exist a $\mathcal B(\mathsf X)$-measurable function $w_1: \mathsf X \rightarrow [1,\infty)$, positive constants $\gamma_1 \in (0,1)$ and $K_1 > 0$ such that $\mathcal R_{x,a}(w_1) \leq \gamma_1 w_1(x) + K_1$.
\end{assumption}

If the above assumption holds and setting $w_0 := w_1^p$ with any $p \in (0,1)$, then for all $f \in \mathscr B_{w_0} \subseteq \mathscr B_{w}$, there exists a constant $K_f$ (depending on $p$ and $\lVert f \rVert_{w_0}$) satisfying $\lvert f(x) \rvert \leq \lVert f \rVert_{w_0} w_0(x) \leq w_1(x) + K_f, \forall x \in \mathsf X.$ We immediately have $$Q_{x,a}[e^f] \leq Q_{x,a}[e^{w_1+K_f}] \leq e^{K_f} e^{\gamma_1 w_1 + K_1} < \infty, \forall (x,a) \in \mathsf K$$ and therefore, the upper envelope for the entropic map in \eqref{eq:ent:upenv} is well-defined. In the following theorem, we show that $w_0 = w^p_1$ with any $p \in (0,1)$ satisfies \ue{1} for the upper envelope of the entropic map with some constants $\gamma_2 \in (0,1)$ and $K_2 > 0$.

\begin{theorem}\label{th:upperenv}
Suppose that Assumption \ref{ass:lya} holds. Let $w_0 := w^p_1$ with $p \in (0,1)$. Then, for any constant $K > 0$, there exist constants $\gamma_2 \in (0,1)$ (depending only on $p$ and $\gamma_1$) and $K_2 > 0$ (depending on $p$, $K$, $\lambda_1$ and $K_1$) such that
\begin{align*}
\sup_{f \in \mathscr B_{w_0}^{(K)}} \frac{Q_{x,a}[e^f w_0]}{Q_{x,a}[e^f]} \leq \gamma_2 w_0(x) + K_2, \forall (x,a) \in \mathsf K.
\end{align*}
\end{theorem}
Since the proof is rather technical, we postpone it to the Appendix.  

Note that by the definition of upper envelope, $\mathcal R_{x, a}(w_0) \leq \sup_{f \in \mathscr B_{w_0}^{(K)}} \frac{Q_{x,a}[e^f w_0]}{Q_{x,a}[e^f]}$. Hence, the above theorem implies immediately the following corollary.
\begin{corollary}\label{cor:lya}
  Suppose that Assumption \ref{ass:lya} holds. Then, for any $p \in (0,1)$, $w_0 := w^p_1$, there exist constants $\hat \gamma_0 \in (0,1)$ (depending on $p$ and $\gamma_1$) and $\hat K_0$ (depending on $p$, $\gamma_1$ and $K_1$) satisfying $\mathcal R_{x, a}(w_0) \leq \hat \gamma_0 w_0(x) + \hat K_0, \forall (x,a) \in \mathsf K.$
\end{corollary}

In summary, if Assumption \ref{ass:lya} holds, then
\begin{enumerate}
 \item by Corollary \ref{cor:lya}, Assumption \lya{1a'} stated in Section \ref{sec:ave:criterion} holds with $w_0$ and constants $\hat \gamma_0$ and $\hat K_0$, and if in addition, the cost function $c$ satisfies $\lvert c\rvert \leq \tilde \gamma_0 w_0 + C_0$ with some $\tilde \gamma_0 \in (0,1-\gamma_0)$ and $C_0 > 0$, then \lya{1} holds;
 \item by Theorem \ref{th:upperenv}, Assumption \ue{1} stated in Section \ref{sec:geo:cont} holds with the same $w_0$ and constants $\gamma_2$ and $K_2$.
\end{enumerate}

\subsection{Doeblin-type conditions}
We introduce the following notation of \emph{level-sets}. For any unbounded nonnegative $\mathcal B(\mathsf X)$-measurable function $w$ and any real number $R \in \mathbb R$, we define $\mathsf B_w(R) :=  \left\{ x \in \mathsf X | w(x) \leq R \right\}$ 
and $\mathsf B_w^c(R)$ its complementary set. We investigate in this section the properties of the entropic map restricted to bounded level-sets. We first introduce the \emph{local Doeblin condition} (see \cite{douc2009forgetting} and references therein) as follows.
\begin{assumption}\label{ass:local}
Let $w_0: \mathsf X \rightarrow [0,\infty)$ be a $\mathcal B(\mathsf X)$-measurable function. For any level-set $\mathsf C := \mathsf B_{w_0}(R)$, $R > 0$, there exist a measure $\mu_{\mathsf C}$ and constants $\lambda_{\mathsf C}^+, \lambda_{\mathsf C}^- > 0$ such that $\mu_{\mathsf C}(\mathsf C) > 0$ and
\begin{align*}
\lambda_{\mathsf C}^- \mu_{\mathsf C}(\mathsf D \cap \mathsf C) \leq Q_{x, a}(\mathsf D \cap \mathsf C) \leq \lambda_{\mathsf C}^+\mu_{\mathsf C}(\mathsf D \cap \mathsf C), \forall x \in \mathsf C, a \in \mathsf A(x), \forall \mathsf D \in  \mathcal B(\mathsf X).
\end{align*}
\end{assumption}

\begin{remark} \label{rm:doeblin}\rm
 This assumption is stronger than the standard Doeblin condition (see Assumption \ref{assp:standard}(ii)). In fact, it is easy to verify that the following two conditions are equivalent:
 \begin{itemize}
\item[(i)] There exist a measure $\mu_{\mathsf C}$ and a constant $\lambda_{\mathsf C}^- > 0$ such that $\mu_{\mathsf C}(\mathsf C) > 0$ and
\begin{align}
  Q_{x,a}(\mathsf D \cap \mathsf C) \geq \lambda_{\mathsf C}^- \mu_{\mathsf C}(\mathsf D \cap \mathsf C), \forall x \in \mathsf C, a \in \mathsf A(x),  \mathsf D \in  \mathcal B(\mathsf X). \label{eq:lowersupport}
\end{align}
\item[(ii)] There exist a probability measure $\mu$ and a constant $\alpha> 0$ such that
  \begin{align}
   Q_{x,a}(\mathsf D) \geq \alpha \mu(\mathsf D),  \forall x \in \mathsf C, a \in \mathsf A(x), \mathsf D \in  \mathcal B(\mathsf X). \label{eq:deoblin_condition}
  \end{align}
\end{itemize}
\end{remark}

\begin{theorem}\label{th:local1}
Suppose Assumption \ref{ass:local} and Assumption \ref{ass:lya} hold. Let $w_1$ be the weight function as in Assumption \ref{ass:lya} and $\mathsf{B} = \mathsf B_{w_0}(R_0)$ be a bounded level-set w.r.t.\ some $R_0 > 0$, where $w_0 := w^p_1$, $p \in (0,1)$. Then for any positive constant $K_0 >0$, there exists a positive constant $K > K_0$ such that
$$\mathcal R_{x,a}(v) - \mathcal R_{y,b}(v) \leq 2 (K - K_0) + \mathcal R_{x,a} (w_0) - \mathcal R_{y,b}(-w_0)$$
holds for all $ x,y \in \mathsf B, a \in \mathsf A(x), b \in \mathsf A(y)$ and $v\in \mathscr B_{1+w_0}$ satisfying $\lvert v \rvert \leq w_0 + K$.
\end{theorem}
\begin{proof}
Let $\mathsf C := \mathsf B_{w_0}(R) \supset \mathsf B = \mathsf B_{w_0}(R_0)$ with $R > R_0$. Let $\mathbf 1_{\mathsf D}(\cdot)$ be the indicator function on $\mathsf X$ for any $\mathsf D \subset \mathsf X$. Then
\begin{align}
 \frac{Q_{x, a} [e^v]}{Q_{y,b}[e^v]} = \frac{Q_{x, a} [e^v \mathbf 1_{\mathsf C}] + Q_{x, a} [e^v \mathbf 1_{\mathsf C^c}]}{Q_{y,b} [e^v \mathbf 1_{\mathsf C}] + Q_{y,b} [e^v \mathbf 1_{\mathsf C^c}]} \leq \frac{Q_{x, a} [e^v \mathbf 1_{\mathsf C}] + Q_{x, a} [e^v \mathbf 1_{\mathsf C^c}]}{Q_{y,b} [e^v \mathbf 1_{\mathsf C}]}. \label{eq:ent:local}
\end{align}
We first consider the second quotient. By $\lvert v \rvert \leq K +w_0$, we obtain
\begin{align*}
 \frac{Q_{x, a}[e^v \mathbf 1_{\mathsf C^c}]}{Q_{y,b}[e^v \mathbf 1_{\mathsf C}]} \leq e^{2 K} \frac{Q_{x, a}[e^{w_0} \mathbf 1_{\mathsf C^c}]}{Q_{y,b}[e^{-w_0} \mathbf 1_{\mathsf C}]} = e^{2 K} \frac{\theta(x, a, \mathsf C) Q_{x, a}[e^{w_0}]}{\theta'(y, b, \mathsf C) Q_{y,b}[e^{-w_0}]}
\end{align*}
where we define 
$
 \theta(x,a, \mathsf C)  := \frac{Q_{x, a}[e^{w_0} \mathbf 1_{\mathsf C^c}]}{Q_{x, a}[e^{w_0}]} \textrm{ and } \theta'(y,b, \mathsf C) :=  \frac{Q_{y, b}[e^{-w_0} \mathbf 1_{\mathsf C}]}{Q_{y, b}[e^{-w_0}]}.
$
By Theorem \ref{th:upperenv}, there exist some constants $\gamma_2 \in (0,1)$ and $K_2 > 0$ such that
\begin{align*}
\theta(x,a,\mathsf C) 
& \leq \lVert \mathbf 1_{\mathsf C^c} \rVert_{w_0} \frac{Q_{x, a} [e^{w_0} w_0]}{Q_{x, a} [e^{w_0}]} 
\leq\lVert \mathbf 1_{\mathsf C^c} \rVert_{w_0} \sup_{\lvert v \rvert \leq w_0} \frac{Q_{x, a} [e^{v} w_0]}{Q_{x, a} [e^{v}]}  \\ 
& \leq \lVert \mathbf 1_{\mathsf C^c} \rVert_{w_0}  (\gamma_2 w_0(x) + K_2).
\end{align*}
Hence,  $\theta(x,a, \mathsf C) \leq  \lVert \mathbf 1_{\mathsf C^c} \rVert_{w_0} \sup_{x \in \mathsf B} (\gamma_2 w_0(x) + K_2) \frac{\gamma_2 R_0 + K_2}{R}.$ Similarly,
\begin{align*}
 \theta'(y, b, \mathsf C) = 1 - \frac{Q_{y, b}[e^{-w_0} \mathbf 1_{\mathsf C^c})}{Q_{y, b}[e^{-w_0}]} \geq 1 - \frac{\gamma_2 R_0 + K_2}{R}.
\end{align*}
Hence, $\sup_{x,y \in \mathsf B, a \in \mathsf A(x), b \in \mathsf A(y)} \frac{\theta(x, a, \mathsf C)}{\theta'(y, b, \mathsf C)} \rightarrow 0$, as $R \rightarrow \infty$, implies that 
for any $K_0 > 0$, we can select sufficiently large $R$ such that
\begin{align}
\ln \frac{\theta(x, a, \mathsf C) }{\theta'(y, b, \mathsf C)} \leq - 2K_0 - \ln 2, \forall x, y \in \mathsf B\, , 
\label{eq:dc}
\end{align}
so that 
\begin{align}
\frac{Q_{x, a}[e^v \mathbf 1_{\mathsf C^c}]}{Q_{y, b}[e^v \mathbf 1_{\mathsf C}]} \leq e^{2(K - K_0) + \mathcal U_{x} (w_0) - \mathcal U_{y}(-w_0) - \ln 2} \label{eq:ent:2}
\end{align}
where $\mathcal U_x (v) := \sup_{a\in A(x)} \mathcal R_{x,a} (v)$ (cf.\ the proof 
of Theorem \ref{th:pe:contraction} in the Appendix). Now we consider the first quotient in \eqref{eq:ent:local}. By Assumption \ref{ass:local}, we immediately have $\frac{Q_{x, a}[e^v \mathbf 1_{\mathsf C}]}{Q_{y, b}[e^v \mathbf 1_{\mathsf C}]} \leq \frac{\lambda_{\mathsf C}^+}{\lambda_{\mathsf C}^-}.$ 
Hence, setting \begin{align}
 K := K_0 + \frac{1}{2}\ln2 + \ln \left(\frac{\lambda_{\mathsf C}^+}{\lambda_{\mathsf C}^-}\right), \label{eq:k'}
\end{align}
we obtain $Q_{x, a}[e^v \mathbf 1_{\mathsf C}]/Q_{y, b}[e^v \mathbf 1_{\mathsf C}] \leq e^{2(K - K_0) + \mathcal U_x(w_0) - \mathcal U_{y}(-w_0) - \ln 2}.$ Together with \eqref{eq:ent:2}, it yields the required inequality: $$\frac{Q_{x, a}[e^{v}] }{Q_{y, b}[e^{v}] }  \leq e^{2(K - K_0) + \mathcal U_x(w_0) -\mathcal U_{y}(-w_0)},$$ where $K$ is chosen according to \eqref{eq:k'}, while $R$ is determined by \eqref{eq:dc}. \quad
\end{proof}

We now investigate the Doeblin-type condition \ue{2} stated in Section \ref{sec:geo:cont} for the upper envelope $\bar{\mathcal R}^{(w,K)}$ of the entropic map.
\begin{proposition}\label{prop:local2}
Let $w: \mathsf X \rightarrow [1,\infty)$ be a $\mathcal B(\mathsf X)$-measurable function and $\mathsf B: = \mathsf B_w(R)$ with some $R > 0$. Suppose Assumption \ref{ass:local} holds. Assume further that $\bar{\mathcal R}^{(w,K)}_{x,a}(w_0) < \infty$ for all $x \in \mathsf B$ and $a \in \mathsf A(x)$. Then there exist a constant $\alpha \in (0,1)$ and a probability measure on $(\mathsf X, \mathcal B(\mathsf X))$ satisfying
\begin{align*}
\bar{\mathcal R}^{(w,K)}_{x,a}(v) - \bar{\mathcal R}^{(w,K)}_{x,a}(u) \geq \alpha \mu[v-u], \forall x \in \mathsf B, a \in \mathsf A(x), v \geq u \in \mathscr B_{1+w_0}.
\end{align*}
\end{proposition}
\begin{proof}
Note that since $\bar{\mathcal R}^{(w,K)}_{x,a}(w_0) < \infty$, we have for all $v \in  \mathscr B_{1+w_0}$, $x \in \mathsf B$ and $a \in \mathsf A(x)$,
\begin{align*}
 \lvert \bar{\mathcal R}^{(w,K)}_{x,a}(v) \rvert \leq  \bar{\mathcal R}^{(w,K)}_{x,a}(\lvert v \rvert) \leq \lVert v \rVert_{1+w_0}  \bar{\mathcal R}^{(w,K)}_{x,a}(1+w_0) < \infty.
\end{align*}

By \eqref{eq:ent:upenv}, we have for all $x \in \mathsf B$ and $v \geq u \in \mathscr B_{1+w_0}$,
\begin{align*}
 \bar{\mathcal R}^{(w,K)}_{x,a}(v) - \bar{\mathcal R}^{(w,K)}_{x,a}(u) = &\sup_{h \in \mathscr B_w^{(K)}} \frac{Q_{x,a}[e^h v]}{Q_{x,a}[e^h]} - \sup_{h' \in \mathscr B_w^{(K)}} \frac{Q_{x,a}[e^{h'} u]}{Q_{x,a}[e^{h'}]} \\
            \geq & \inf_{h' \in \mathscr B_w^{(K)}}  \frac{Q_{x,a} \left[e^{h'} (v - u) \right]}{Q_{x,a}[e^{h'}]}.
\end{align*}
By Remark \ref{rm:doeblin}, Assumption \ref{ass:local} implies that there exist a probability measure $\mu_{\mathsf B}$ and $\alpha_{\mathsf B}$ such that $Q_{x,a}[v] \geq \alpha_{\mathsf B} \mu_{\mathsf B}[v]$ for all nonnegative measurable functions $v$. Hence, for all $x\in \mathsf B$ and $h' \in \mathscr B_w^{(K)}$, we have
\begin{align*}
 \frac{Q_{x,a} \left[e^{h'} (v - u) \right]}{Q_{x,a}[e^{h'}]} \geq  & \frac{\alpha_{\mathsf B} \mu_{\mathsf B}\left[e^{h'} (v - u) \right]}{Q_{x,a}[e^{Kw}]} 
 \geq \frac{\alpha_{\mathsf B} \mu_{\mathsf B}\left[e^{-Kw} (v - u) \right]}{\max_{x \in \mathsf B} Q_{x,a}[e^{Kw}]} \\ =& \frac{\alpha_{\mathsf B} \mu_{\mathsf B}[e^{-Kw}]}{\max_{x \in \mathsf B} Q_{x,a}[e^{Kw}]} \frac{\mu_{\mathsf B}\left[e^{-Kw} (v - u) \right]}{\mu_{\mathsf B}\left[e^{-Kw}  \right]}.
\end{align*}
Hence, $\alpha := \dfrac{\alpha_{\mathsf B} \mu_{\mathsf B}[e^{-Kw}]}{\max_{x \in \mathsf B} Q_{x,a}[e^{Kw}]} \textrm{ and the probability measure }
d \mu := \dfrac{e^{-Kw} d\mu_{\mathsf B}}{\int e^{-Kw} d\mu_{\mathsf B}}$
are the required constant and probability measure respectively. \quad
\end{proof}

The following theorem shows that applying the entropic map, along with an additional growth condition for cost functions (see \eqref{eq:cost} below), the existence of Lyapunov function stated in Assumption \ref{ass:lya} and the local Doeblin condition stated in Assumption \ref{ass:local} are sufficient for Assumption \ref{asmp:subset} and \ref{assp:upper:envelope}. 
\begin{theorem}\label{th:entropic:summary}
 Let $\mathcal R$ be the entropic map with $\lambda = 1$. Suppose Assumption \ref{ass:lya} and \ref{ass:local} hold with a weight function $w_1$. If the cost function $c$ satisfies
 \begin{align}
   \lvert c(x,a) \rvert \leq C w_1^q(x), \forall (x,a) \in \mathsf K, \textrm{ with some } q \in (0,1), \label{eq:cost}
 \end{align}
 then Assumption \ref{asmp:subset} holds with $w_0 = w^p_1$ for any $p \in (q,1)$, and some $ K > 0$, and Assumption \ref{assp:upper:envelope} holds with $w = 1 + K^{-1} w_0$. 
\end{theorem}
\begin{proof}
Fix one $p \in (q,1)$ and let $w_0 = w_1^p$. Then by Corollary \ref{cor:lya}, there exists $\hat \gamma_0 \in (0,1)$ and $\hat K_0 > 0$ satisfying 
$\mathcal U_x(w_0) \leq \hat \gamma_0 w_0(x) + \hat K_0.$ Choosing one $\gamma^{(c)}_0 \in (0,1-\hat \gamma_0)$, there exists a constant $K_0^{(c)} > 0$ satisfying $C w_1^q(x) \leq \gamma^{(c)}_0 w_0(x) + K_0^{(c)}$. Hence, Assumption \ref{asmp:subset} \lya{1} holds with $\gamma_0 := \hat \gamma_0 + \gamma^{(c)}_0 \in (0,1)$ and $K_0 := \hat K_0 + K_0^{(c)}$. Due to Remark \ref{rm:doeblin}, Assumption \ref{asmp:subset} \lya{2} holds with some constant $K > K_0$. Next, by Theorem \ref{th:inv_bound} and \ref{th:upperenv}, Assumption \ref{assp:upper:envelope} \ue{1} holds with $w := 1 + K^{-1} w_0$. Assumption \ref{assp:upper:envelope} \ue{2} holds due to Proposition \ref{prop:local2}. \quad
\end{proof}

\subsection{Discussion}
Most of the existing literature on risk-sensitive MCPs applying the entropic map, considers finite or countable state spaces \cite{avila1998controlled, borkar2002risk,coraluppi2000mixed,fleming1997risk, hernandez1996risk, marcus1997risk, cavazos2010optimality} or bounded cost functions \cite{bauerle2013more}. 
We consider in this paper a more general setting: Borel state-action spaces and unbounded cost functions. The literature under the same setting includes \cite{masi2000infinite,kontoyiannis2003spectral, kontoyiannis2005large, jaskiewicz2007} and \cite{di2008infinite}, with which we provide below a detailed comparison. 

Among others, \citet{kontoyiannis2005large} developed (see also their earlier work \cite{kontoyiannis2003spectral}) a spectral theory of multiplicative Markov processes, where the Poisson equation w.r.t.\ the entropic map (called multiplicative Poisson equation in \cite{kontoyiannis2005large}) plays the central role. Though our assumptions are less general than the assumptions stated in \cite{kontoyiannis2005large, kontoyiannis2003spectral}, our proof that generalizes the Hairer-Mattingly approach \cite[]{hairer2011yet} is conceptually simpler than the one provided in \cite{kontoyiannis2005large, kontoyiannis2003spectral}, and can also be applied to other types of valuation maps. 

To guarantee the existence of a solution to the optimal control problem, \citeauthor{jaskiewicz2007} stated a set of more general conditions in \cite{jaskiewicz2007} than the set stated in this section. However, Condition (B) in \cite{jaskiewicz2007} which assumes the boundedness of iterations is very difficult to verify. In contrast to it, we provide in this section a set of more verifiable conditions that guarantees this boundedness (see Theorems \ref{th:entropic:summary} and \ref{th:inv_bound}) and hence Condition (B) in \cite{jaskiewicz2007}. It is also worth to mention that the cost functions allowed in \cite{jaskiewicz2007} are assumed to be lower bounded (but not upper bounded), which is not as general as our setting.



The assumption (A4) in Section 4 of \citet{di2008infinite} (see also their earlier work \cite{masi2000infinite}) requires a positive continuous density, i.e., there exists a positive function $q$ satisfying $Q(\diff y|x,a) = q(x,a,y) \mu(\diff y)$ for some reference probability measure $\mu$, which implies the local Doeblin condition in Assumption \ref{ass:local}. Hence, our assumption is more general than its counterpart in \cite{di2008infinite}. The assumption (A3) set in Section 3 of \cite{di2008infinite} for the cost function $c$ is implicit and difficult to be verified. On the contrary, the sufficient growth condition \eqref{eq:cost} for $c$, is explicit in form of the weight function $w_1$ w.r.t.\ the entropic map. Note that, in the example provided by \cite{di2008infinite}, the assumption (A3) is also verified with the help of a weight function.

Finally, as an advantage, in comparison with \cite{kontoyiannis2005large, jaskiewicz2007} and \cite{di2008infinite}, the convergence rate of iterations towards the solution to the Poisson equation is explicitly specified by $\bar \alpha$ in Theorem \ref{th:pe:contraction} under the chosen seminorm.

\section{Examples}
 \label{sec:examples}
 In this section, we utilize the following example taken from \cite[Section 6]{di2008infinite} as a canonical example of MCPs. Let $\lVert \cdot \rVert$ denote the Euclidean norm, and $\lVert \cdot \rVert_\infty$ the sup-norm. 
\subsection{A canonical example of MCP}\label{sec:canonical} 
 Let $\mathsf X = \mathbb R^d$ and $\mathsf A$ be a compact subset of a Euclidean space. Consider the following discretized ergodic diffusion $\{ x_t \in \mathbb R^d\}$:
\begin{align}
 X_{t+1} = A X_t + b(X_t, A_t) + D(X_t, A_t) W_t, \textrm{ where } \label{eq:diff}
\end{align}
\begin{itemize}
 \item  $\{ W_t \in \mathbb R^d\}$ is a sequence of i.i.d.\ standard white noise; 
 \item $D: \mathsf K \rightarrow \mathbb R^{d\times d}$ is a continuous bounded matrix-valued function which is uniformly elliptic, i.e., there exists a constant $L>0$ such that
\begin{align}
 L^{-1} \lVert \xi \rVert^2 \leq \xi^\top D^\top(x,a) D(x,a) \xi \leq L \lVert \xi \rVert^2, \forall (x,a) \in \mathsf K, \xi \in \mathbb R^d; \label{eq:D}
\end{align}
\item $b: \mathsf K \rightarrow \mathbb R^d$ is a continuous bounded vector function, i.e., there exists a positive constant $B > 0$ such that $ \lVert b(x,a) \rVert^2 \leq B, \forall (x,a) \in \mathsf K$; and
\item $A$ is a matrix satisfying that there exists a constant $\tilde \gamma \in (0,1)$ such that $\xi^\top A^\top A \xi \leq \tilde \gamma \lVert \xi \rVert^2$, $\forall \xi \in \mathbb R^d$. 
\end{itemize}
Then the transition kernel $Q(\diff y |x,a)$ has the following density w.r.t.\ the Lebesgue measure,
\begin{align}
 q(y | x,a) = (2\pi)^{-d/2} \lvert \Sigma \rvert^{1/2} e^{-\frac{1}{2} (y - A x - b)^\top \Sigma (y- A x -b)}, \textrm{ with } \Sigma := (D D^\top)^{-1}. \label{eq:den}
\end{align}

\begin{remark}\label{rm:weak}\rm
 Note that since for each $x$ and $y$, $a \mapsto q(y | x,a)$ is continuous on $\mathsf A(x)$,  $Q_{x,a}$ is weakly continuous on $\mathsf A(x)$, i.e., for each $x \in \mathsf X$ and bounded $v \in \mathscr B$, $a \mapsto Q_{x,a}(v)$ is continuous on $\mathsf A(x)$.
\end{remark}



\subsection{Entropic maps}\label{sec:ex:entropic}
Let $\mathcal R$ be the entropic map defined in Example \ref{ex:entropic} with $\lambda = 1$. By Theorem \ref{th:entropic:summary}, it is sufficient to verify the existence of Lyapunov function stated in Assumption \ref{ass:lya} and the local Doeblin condition stated in Assumption \ref{ass:local}. 

Among them, the local Doeblin condition is satisfied since the transition kernel $Q$ has a positive continuous density function $q$ w.r.t.\ the Lebesgue measure. 

It remains to verify the the existence of Lyapunov function stated in Assumption \ref{ass:lya}. Take one $\gamma \in (\tilde \gamma, 1)$ and consider the following weight function 
\begin{align}
 \hat w_1(x) = \frac{\epsilon}{2} \lVert x \rVert^2, \textrm{ with some positive } \epsilon \leq \frac{\gamma - \tilde \gamma}{\gamma} L^{-1} < L^{-1}. \label{eq:w1}
\end{align}
Hence, $\Sigma(x,a) - \epsilon I$ is positive definite for all $(x,a) \in \mathsf K$. We show that $\hat w_1$ is a Lyapunov function w.r.t.\ the entropic map as follows. By setting $\tilde x := A x + b$, we obtain
\begin{align*}
 \int Q(\diff y|x,a) e^{\hat w_1(y)} = & (2\pi)^{-d/2} \lvert \Sigma \rvert^{1/2} \int e^{-\frac{1}{2} \left( y^\top (\Sigma -\epsilon I) y - 2 y^\top \Sigma \tilde x^\top + \tilde x^\top \Sigma \tilde x \right)} \diff y \\
 = & \frac{\lvert \Sigma \rvert^{1/2}}{\lvert \Sigma - \epsilon I \rvert^{1/2}} e^{\frac{1}{2} \tilde x^\top \Sigma \left( (\Sigma - \epsilon I)^{-1}  - \Sigma^{-1} \right) \Sigma \tilde x}, \textrm{ which yields} 
\end{align*}
\begin{align}
 Q_{x,a} [e^{\hat w_1}] = \frac{\lvert \Sigma \rvert^{1/2}}{\lvert \Sigma - \epsilon I \rvert^{1/2}} e^{\frac{1}{2} (A x + b)^\top \Sigma \left( (\Sigma - \epsilon I)^{-1}   - \Sigma^{-1} \right) \Sigma (A x + b)}. \label{eq:hatw1}
\end{align}
By \eqref{eq:D} and the choice of $\epsilon$ in \eqref{eq:w1}, we have
\begin{align*}
 \frac{1}{2} x^\top A^\top \Sigma \left( (\Sigma - \epsilon I)^{-1}   - \Sigma^{-1} \right) \Sigma A x \leq \frac{\gamma \epsilon}{2} \lVert x \rVert^2 = \gamma \hat w_1(x), \forall (x,a) \in \mathsf K.
\end{align*}
Finally, due to the uniform boundedness of $b$ and $\frac{\lvert \Sigma \rvert^{1/2}}{\lvert \Sigma - \epsilon I \rvert^{1/2}}$, we can always select some $\gamma_1 \in (\gamma,1)$ and $\hat K > 0$ such that
\begin{align*}
 \ln \left( \int Q(\diff y|x,a) e^{\hat w_1(y)} \right) \leq \gamma_1 \hat w_1(x) + \hat K, \forall (x,a) \in \mathsf K. 
\end{align*}
Hence, Assumption \ref{ass:lya} is verified with $w_1 := 1 + \hat w_1$, $\gamma_1$ and $K_1 := \hat K_1 + 1 - \gamma$. 

A final remark is addressed to Assumption \os{3}, where the lower semicontinuity can be in fact strengthen to be continuity in this example. Note that for each $x \in \mathsf X$, (i) $v \in \mathscr B$, $a \mapsto Q_{x,a}[v]$ is continuous on $\mathsf A(x)$ (see Remark \ref{rm:weak}); and (ii) by \eqref{eq:hatw1}, $a \mapsto Q_{x,a}[\hat W]$ is continuous on $\mathsf A(x)$, where $\hat W := e^{\hat w_1}$. Hence, by \cite[Lemma 8.3.7(a)]{hernandez1999further}, for each $v \in \mathscr B_{\hat W}$, $a \mapsto Q_{x,a}[v]$ is continuous on $\mathsf A(x)$. Note that if $v \in \mathscr B_{w_0}$ satisfies $\lvert v \rvert \leq w_0 + K$, where $w_0 = w_1^p$, $p \in (0,1)$ (see Theorem \ref{th:entropic:summary}), then $e^v \in \mathscr B_{\hat W}$. Hence, we obtain immediately the required weak continuity.



\subsection{Mean-semideviation trade-off}
Recall that in Example \ref{ex:meanvar} the mean-semi\-deviation trade\-off is defined as
\begin{align*}
 \mathcal R_{x,a}(v) := Q_{x,a}[v] + \lambda \left( Q_{x,a} \left[\left(v - Q_{x,a}[v]\right)^r_+\right] \right)^{1/r}, \lambda \in (0,1). 
\end{align*}
Note that since $\mathcal R$ is coherent \cite{ruszczynski2006optimization}, Assumptions \ref{asmp:subset} and \ref{assp:upper:envelope} can be reduced to Assumption \ref{asmp:coherent} (see Remark \ref{rm:contraction}). We consider below the case $r=2$.

\paragraph{Verification of \lya{1a'}}
Consider the canonical example. Let $\tilde x := Ax + b(x,a)$ and set $w_0(x) = \lVert x \rVert^2$. Hence, given $(x,a)$, $Y := \tilde x + D(x,a) W$, where $W$ is a $d$-dimensional white noise, follows a multivariate normal distribution with mean $\tilde x$ and covariance matrix $DD^\top$. We have then
\begin{align*}
 Q_{x,a}[w_0] = & Q_{x,a}[\lVert Y \rVert^2] = \lVert\tilde x \rVert^2 + tr(DD^\top), \textrm{ and} \\
 Q_{x,a}[w_0 - Q_{x,a}[w_0]]^2 = & Q_{x,a}[w_0^2] - \left( Q_{x,a}[w_0] \right)^2 \\
  = & 4 \mathbb E[\tilde x^\top D W]^2 + 4 \mathbb E[\tilde x^\top D W \lVert D W \rVert^2] \\
  & + \mathbb E[\lVert DW \rVert^4] - \left( \mathbb E[\lVert DW \rVert^2] \right)^2.
\end{align*}
Here, the expectation $\mathbb E$ is taken over the white noise $W$. By the assumptions on $A$ and $b$, there exist constants $\epsilon \in (\tilde \gamma, 1)$, $L_0 > 0$ such that $\lVert\tilde x \rVert^2 \leq \epsilon \lVert x \rVert^2 + L_0.$ By the assumption on $D$, there exists a constant $L_1 > 0$ such that
\begin{align*}
 Q_{x,a}[w_0 - Q_{x,a}[w_0]]^2 \leq L_1 (\lVert \tilde x \rVert^2 + 1).
\end{align*}
Hence, we have
\begin{align*}
 \mathcal R_{x,a}(w_0) = & Q_{x,a}[w_0] + \lambda \sqrt{Q_{x,a} \left[\left(w_0 - Q_{x,a}[w_0]\right)_+^2\right] } \\
 \leq & Q_{x,a}[w_0] + \lambda \sqrt{Q_{x,a} \left[\left(w_0 - Q_{x,a}[w_0]\right)^2\right] } \\
  \leq & \epsilon \lVert x \rVert^2 + L_0 + tr(DD^\top) + \lambda\sqrt{ L_1 (\lVert \tilde x \rVert^2 + 1)}
\end{align*}
which implies that there exist positive constants $L_0'$ and $L_1'$ such that
\begin{align}
 \mathcal R_{x,a}(w_0) = \epsilon w_0(x) + L_0' + \lambda L_1'\sqrt{w_0(x) + 1} \label{eq:msemid}
\end{align}
Taking one $\gamma_0 \in (\epsilon, 1)$, since $t \mapsto (\epsilon-\gamma_0)t + L_0' + \lambda L_1' \sqrt{1 + t}$ is concave in $[0, \infty)$ and its maximum is attained at some constant $K_0$, we have
$
 \mathcal R_{x,a}(w_0) \leq \gamma_0 w_0(x) + K_0,
$ which verifies \lya{1a'}.



\paragraph{Verification of \lya{2'}} One easy extension of the subgradient calculation presented by \citet{svindland2009subgradients} (see also \cite[Section 6]{Shen2013}) yields
\begin{align*}
 \mathcal R_{x,a}(v) - \mathcal R_{x,a}(u) \geq \int g(x,a,u, y) (v(y) - u(y)) Q_{x,a}(\diff y), \forall v \geq u \in \mathscr B_{1+w_0}, \textrm{ with}
\end{align*}
\begin{align*}
 g(x,a,u) := \left\{ \begin{array}{ll}
                   1 & \textrm{if $u$ is constant}\\
                1 - \lambda \dfrac{Q_{x,a} \left[ ( u - Q_{x,a}[u] )_+ \right] - ( u - Q_{x,a} u )_+}{\sqrt{Q_{x,a}[( u - Q_{x,a}[u] )_+^{2}]}} & \textrm{otherwise}
                  \end{array}
 \right.
\end{align*}
On the other hand, 
\begin{align*}
 Q_{x,a} \left[ ( u - Q_{x,a} u )_+ \right] - ( u - Q_{x,a} [u] )_+ \leq  Q_{x,a} \left[ ( u - Q_{x,a} u )_+ \right] \leq \sqrt{Q_{x,a} \left[ ( u - Q_{x,a} u )_+^2 \right]}
\end{align*}
implies that $g(x,a,u) \geq 1-\lambda > 0$. Note that in the canonical MCP, $Q_{x,a}(\cdot)$ is supported by a probability measure $\mu$ on any bounded level-set. Hence,
\begin{align*}
 \mathcal R_{x,a}(v) - \mathcal R_{x,a}(u) \geq (1-\lambda) \int (v(y) - u(y)) Q_{x,a}(\diff y) \geq \alpha (1-\lambda) \mu[v- u]
\end{align*}
verifies \lya{2'}.

\paragraph{Verification of \os{3}} The lower semicontinuity can be in fact strengthen to be continuity in this example. Note that for each $x \in \mathsf X$, (i) $Q_{x,a}$ is weakly continuous on $\mathsf A(x)$ (see Remark \ref{rm:weak}) and (ii) by  $a \mapsto Q_{x,a}[w_0]$ is continuous on $\mathsf A(x)$. By \cite[Lemma 8.3.7(a)]{hernandez1999further}, for each $v \in \mathscr B_{1+w_0}$, $a \mapsto Q_{x,a}[v]$ is continuous on $\mathsf A(x)$. It remains to show that for each $x \in \mathsf X$ and $v \in \mathscr B_{1+w_0}$, $a \mapsto \sqrt{Q_{x,a} [(v - Q_{x,a}[v])_+^2]}$ is continuous on $\mathsf A(x)$. Since $t \rightarrow \sqrt{t}$ is continuous on $[0,\infty)$, it is sufficient to show the continuity of the mapping $a \mapsto Q_{x,a} [(v - Q_{x,a}[v])_+^2]$. Now fix $x \in \mathsf X$ and $v \in \mathscr B_{1+w_0}$, and let $\{a_i \in \mathsf A(x) \} $ converging to $a \in \mathsf A(x)$ as $i \rightarrow \infty$. Setting $\mu_i = Q_{x,a_i}$ and $\mu = Q_{x,a}$, we have for each $i$
\begin{align}
  &\left\lvert \mu\left[(v-\mu[v])_+^2 \right] - \mu_i\left[(v-\mu_i[v])_+^2 \right] \right\rvert \nonumber \\
 \leq & \left\lvert \mu\left[(v-\mu[v])_+^2  \right] - \mu_i\left[(v-\mu[v])_+^2  \right] \right\rvert + \left\lvert \mu_i\left[(v-\mu[v])_+^2  \right] - \mu_i\left[(v-\mu_i[v])_+^2  \right]\right\rvert \label{eq:msemi}
\end{align}

It is easy to verify that for each $x \in \mathsf X$, $a \mapsto Q_{x,a}[w_0^2]$ is also continuous on $\mathsf A(x)$. Hence, by \cite[Lemma 8.3.7(a)]{hernandez1999further}, for each $v \in \mathscr B_{1+w_0^2}$, $x \in \mathsf X$, $a \mapsto Q_{x,a}[v]$ is continuous on $\mathsf A(x)$. This implies that the first term in \eqref{eq:msemi} converges to 0 as $i \rightarrow \infty$, since for each $v \in \mathscr B_{1+w_0}$, $(v-\mu[v])_+^2 \in \mathscr B_{1+w_0^2}$.

Now we consider the second term in \eqref{eq:msemi}. Note that for any $x, y \in \mathbb R$, we have $\lvert (x)_+ - (y)_+ \rvert \leq \lvert x - y \rvert$. Hence, for any $x \in \mathsf X$,
\begin{align*}
 \left\lvert (v(x)-\mu[v])_+^2 - (v(x)-\mu_i[v])_+^2 \right\rvert \leq \left\lvert \mu[v] - \mu_i[v] \right\rvert \left\lvert (v(x)-\mu[v])_+ + (v(x)-\mu_i[v])_+ \right\rvert.
\end{align*}
It yields
\begin{align*}
&\left\lvert \mu_i\left[(v-\mu[v])_+^2  \right] - \mu_i\left[(v-\mu_i[v])_+^2  \right]\right\rvert\\
\leq & \lVert (v-\mu[v])_+^2 - (v-\mu_i[v])_+^2 \rVert_{1+w_0} \mu_i[1+w_0] \\
 \leq & \left\lvert \mu[v] - \mu_i[v] \right\rvert \left\lVert (v-\mu[v])_+ + (v-\mu_i[v])_+ \right\rVert_{1+w_0} (\epsilon w_0(x) + L_0' + 1) \\ \leq & \left\lvert \mu[v] - \mu_i[v] \right\rvert \left(2 \lVert v \rVert_{1+w_0} + 2 (\epsilon w_0(x) + L_0' + 1)\right) (\epsilon w_0(x) + L_0' + 1),
\end{align*}
where in the last two inequalities we use the fact that $\mu_i[w_0] \leq \epsilon w_0(x) + L_0'$ (cf.\ \eqref{eq:msemid}),  $\forall i \in \mathbb N$. Finally, by the fact that $\forall v \in \mathscr B_{1+w_0}$, $\mu[v] \rightarrow \mu_i[v]$ as $i \rightarrow \infty$, we obtain the convergence of the second term in \eqref{eq:msemi}.

\subsection{Utility-based shortfall}\label{sec:shortfall}
Consider the utility-based shortfall defined in \eqref{eq:shortfall}
\begin{align}
 \mathcal R_{x,a}(v) = \sup \left\{ m \in \mathbb R \mid \int_{\mathsf X} u(v(y) - m ) Q_{x,a}(\diff y) \geq 0 \right\}, \label{eq:ubsf:def}
\end{align}
under the assumption that $u$ is increasing and there exist constants $l$ and $L$ satisfying $0 < l \leq 1 \leq L <\infty$ and 
\begin{align}
 l \leq \frac{u(x) - u(y)}{x -y} \leq L, \forall x, y \in \mathbb R. \label{eq:ubsf:Lipschitz}
\end{align}
In other words, for each $x,y \in \mathbb R$, we have $u(x) - u(y) = \delta(x,y) (x-y)$, with some  $\delta(x,y) \in [l,L].$

\begin{remark}\rm
 Note that $u$ is not required to be convex, nor concave. Hence, the induced risk preference can be mixed and is, therefore, useful for quantifying human behavior \cite{shen2013b}. One example of $u$ that satisfies the assumption we made above is a piecewise linear function with slopes upper bounded by $L$ and lower bounded by $l$.  
\end{remark}

\citet[Proposition 4.104]{follmer2004stochastic} show that if $u$ is continuous and strictly increasing, the optimal $m^*$ is obtained when the equality holds, i.e., for each $(x,a) \in \mathsf K$, $m^*(x,a) := \mathcal R_{x,a}(v)$ satisfies 
\begin{align}
 \int u(v(y) - m^*(x,a)) Q_{x,a}(\diff y) = 0. \label{eq:ubsf}
\end{align}
Let $m = \mathcal R(v)$ and $m' = \mathcal R(v')$. Hence, for each $(x,a) \in \mathsf K$,
 \begin{align*}
  \int u(v(y) - m(x,a)) Q_{x,a}(\diff y) = \int u(v'(y) - m'(x,a)) Q_{x,a}(\diff y) = 0. 
 \end{align*}
Let $k := (x,a)$. Then, 
 \begin{align*}
  0 = & \int u(v(y) - m(k)) Q_k(\diff y) - \int u(v'(y) - m'(k)) Q_{k}(\diff y) \\
  = &  \int \delta(v,v',k,y) (v(y) - v'(y) - m(k) + m'(k)) Q_k(\diff y),
\end{align*}
where $\delta(v,v',k,y) := \frac{u(v(y) - m(k)) - u(v'(y) - m'(k))}{v(y) - v'(y) - m(k) + m'(k)} \in [l, L]$. Hence,
\begin{align}
 \left(m(k) - m'(k)\right) \int \delta(v,v',k,y) Q_k(\diff y) = \int \delta(v,v',k,y) (v(y) - v'(y)) Q_k(\diff y). \label{eq:ubsf:v}
\end{align}

\paragraph*{Verification of Assumption \ref{asmp:subset}} Let $w_0(x) = e^{\epsilon \lVert x \rVert^2}$ be a weight function and we have shown in Section \ref{sec:ex:entropic} that it satisfies
$ Q_{x,a}[w_0] \leq C (w_0(x))^\gamma$, for some $C > 0$ and $\gamma \in (0,1).$ First, taking $v = w_0$ and $v' = 0$ in \eqref{eq:ubsf:v}, we have $m' = 0$ and $\mathcal R_{x,a}(w_0) = m(x) \leq \frac{L}{l} Q_{x,a}[w_0] \leq \frac{L}{l} C (w_0(x))^\gamma.$
Second, taking $v = 0$ and $v' = -w_0$ in \eqref{eq:ubsf:v}, we have $m = 0$ and 
\begin{align}
 -\mathcal R_{x,a}(-w_0) = -m'(x,a) \leq \frac{L}{l} Q_{x,a}[w_0] \leq \frac{L}{l} C (w_0(x))^\gamma. \label{eq:ubsf:derive}
\end{align}
Hence, given a cost function $c$ satisfying $\lvert c(x,a) \rvert \leq C' (1+(w_0(x))^{\gamma'}), \forall (x,a) \in \mathsf K$, with some $\gamma' \in (0,1)$ and $C' > 0$, we can always find a sufficiently large $K_0$ such that
\begin{align}
 \left(c(x,a) + \mathcal R_{x,a}(w_0)\right) \vee \left(-c(x) - \mathcal R_{x,a}(-w_0)\right) \leq \gamma_0 w_0(x) + K_0, \forall (x,a) \in \mathsf K.
\end{align}
This verifies Assumption \ref{asmp:subset} \lya{1}.

Take $v' = w_0 + K$ in \eqref{eq:ubsf:v}, where $K$ will be specified later. Due to the assumption $v \leq w_0 + K$, we have $m \leq m'$ and for any $k = (x,a) \in \mathsf K$, 
\begin{align*}
 L (m(k) - m'(k)) \leq & \left(m(k) - m'(k)\right) \int \delta(v,v',k,y) Q_k(\diff y) \nonumber\\
 \leq & \int \delta(v,v',k,y) (v(y) - v'(y)) Q_k(\diff y) \nonumber\\
 \leq & l \int (v(y) - v'(y)) Q_k(\diff y) = l \int (v(y)- w_0(y) - K) Q_k(\diff y) \nonumber\\
 \textrm{ which implies } \quad &  \mathcal R_k(v) - \mathcal R_k(w_0+ K) \leq \frac{l}{L} \int (v(y)- w_0(y) - K) Q_k(\diff y). 
\end{align*}
Analogously we obtain $\mathcal R_{k}(-w_0- K) - \mathcal R_{k}(v) \leq \frac{l}{L} \int (-w_0- K -v(y)) Q_k(\diff y).$ 
On the other hand, for the canonical MCP, we have $Q_{x,a}(\cdot) \geq \alpha \mu(\cdot)$ with some probability measure $\mu$ and $\alpha > 0$ for each $x$ in any bounded level-set and $a \in \mathsf A(x)$. Hence, 
$\frac{l}{L} Q_{x,a}[ w_0 + K - v] + \frac{l}{L} Q_{x',a'}[ w_0(y) + K + v] \geq \frac{l}{L} 2 \alpha K$
yields $$\mathcal R_{x,a}(w+K) - \mathcal R_{x,a}(v) + \mathcal R_{x',a'}(v) - \mathcal R_{x',a'}(-w_0-K) \geq 2 \frac{\alpha l}{L} K.$$
Therefore, taking $K := \frac{L}{\alpha l} K_0$, the Assumption \ref{asmp:subset} \lya{2} holds. 

\paragraph*{Verification of Assumption \ref{assp:upper:envelope}} By \eqref{eq:ubsf:v}, we have
\begin{align}
 \mathcal R_{x,a}(v) - \mathcal R_{x,a}(v') \leq & \frac{\int \delta(v,v',x,a,y) Q_{x,a}(\diff y)(v(y)-v'(y))}{\int \delta(v,v',x,a, y) Q_{x,a}(\diff y)} \nonumber \\
 \leq & \sup_{\delta: l \leq \delta(x,a,y) \leq L} \frac{\int \delta(x,a,y) Q_{x,a}(\diff y)(v(y)-v'(y))}{\int \delta(x,a, y) Q_{x,a}(\diff y)} = : \bar{\mathcal R}_{x,a}(v-v') \label{eq:ubsf:upper}
\end{align}
It is easy to see that $\bar{\mathcal R}$ is a coherent risk map. Note that
\begin{align*}
 \bar{\mathcal R}_{x,a}(w_0) = \sup_{\delta: l \leq \delta(x,a,y) \leq L} \frac{\int \delta(x,a,y) w(y) Q_{x,a}(\diff y)}{\int \delta(x,a,y) Q_{x,a}(\diff y)} \leq \frac{L}{l} Q_{x,a}[w_0]
\end{align*}
which implies that $w_0$ satisfies \ue{1} with some constants $\gamma \in (0,1)$ and $\bar K > 0$. 

On the other hand, for any $v\geq v' \in \mathscr B_{1+w_0}$,
\begin{align}
 \bar{\mathcal R}_{x,a}(v) - \bar{\mathcal R}_{x,a}(v') \geq & \inf_{\delta: l \leq \delta(x,a, y) \leq L} \frac{\int \delta(x,a, y) Q_{x,a}(\diff y)(v(y) -v'(y))}{\int \delta(x,a,y) Q_{x,a}(\diff y)}  \nonumber \\
 \geq & \frac{l}{L} Q_{x,a}[v-v'] \geq \frac{\alpha l}{L} \mu[v-v'], \label{eq:ubsf:doeblin}
\end{align}
holds for all $x$ in any bounded level-set and $a \in \mathsf A(x)$. Hence, \ue{2} holds.

\paragraph*{Verification of \os{3}} The lower semicontinuity can be in fact strengthen to be continuity in this example. Let $v \in \mathscr B_{1+w_0}$ satisfying $\lvert v \rvert \leq w_0 + K$. Fix $x \in \mathsf X$, and let $\{ a_i \}$ be a sequence of actions in $\mathsf A(x)$ converging to $a$. Set $\mu_i = Q_{x, a_i}$, $\mu = Q_{x,a}$, $m_i = \mathcal R_{x,a_i}(v)$ and $m = \mathcal R_{x,a}(v)$. It is therefore to show $m_i$ converges to $m$. Indeed, by \eqref{eq:ubsf}, we have
\begin{align*}
 0 = & \int u(v(y) - m_i) \mu_i(\diff y)  - \int u(v(y) - m) \mu(\diff y) \\
  = & \int \left(  u(v(y) - m_i) - u(v(y) - m) \right) \mu_i(\diff y)  + \int u(v(y) - m) \left( \mu_i(\diff y) - \mu(\diff y)\right) \\
  = & \int \delta_i(y) (m - m_i) \mu_i(\diff y) + \int u(v(y) - m) \left( \mu_i(\diff y) - \mu(\diff y)\right) \\
  = & (m-m_i) \int \delta_i(y) \mu_i(\diff y) + \int u(v(y) - m) \left( \mu_i(\diff y) - \mu(\diff y)\right) 
\end{align*}
where $\delta_i(y) := \frac{u(v(y) - m_i) - u(v(y) - m)}{m_i - m} \in [l, L]$. Hence,
\begin{align*}
 \lvert m - m_i \rvert = \frac{\lvert \int u(v(y) - m) \left( \mu_i(\diff y) - \mu(\diff y)\right)  \rvert}{\int \delta_i(y) \mu_i(\diff y)} \leq \frac{1}{l} \lvert \int u(v(y) - m) \left( \mu_i(\diff y) - \mu(\diff y)\right)  \rvert.
\end{align*}
Note that by $\lvert u(v(y) - m) \rvert = \lvert u(v(y) - m) - u(0)\rvert \leq L \lvert v(y) - m \rvert$, we have
\begin{align*}
 \lVert u(v(\cdot) - m) \lVert_{1+ K^{-1} w_0} \leq L \left( \lVert v \rVert_{1+K^{-1}w_0} + m \right) 
\end{align*}
and hence $u(v(\cdot) - m) \in \mathscr B_{1+ K^{-1} w_0}$. On the other hand, since for each $x \in \mathsf X$, (i) $a \mapsto Q_{x,a}[v]$ is continuous for each $v \in \mathscr B$ (see Remark \ref{rm:weak}) and (ii) $a \mapsto Q_{x,a}[w_0]$ is continuous, by Lemma 8.3.7(a) in \cite{hernandez1999further}, we have for any $v \in \mathscr B_{1+K^{-1}w_0}$, $\mu_i[v] \rightarrow \mu[v]$ as $i \rightarrow \infty$. This implies the convergence of $m_i$ to $m$.



\section*{Appendix}

\renewcommand{\thesection}{A}


\paragraph{Proof of Theorem \ref{th:pe:contraction}}
Define $w' := 1 + \beta w_0$ for some positive $\beta \in \mathbb R_+$, whose value will be specified later. Suppose $\lVert v-u \rVert_{s,w'} = A \geq 0$. Due to Lemma \ref{lm:semi} and the fact that adding any constant to $v$ and $u$ will not change the values of both sides of the required inequality, we may assume that $\lVert v-u \rVert_{w'} = A$. Define
\begin{align*}
 \bar{\mathcal U}_x(v) := \sup_{a \in \mathsf A(x)} \bar{\mathcal R}^{(w,K)}_{x,a}(v), x \in \mathsf X.
\end{align*}
It is easy to verify that $\bar{\mathcal U}_x(\cdot)$ is monotone, translation invariant, centralized and coherent on $\mathscr B_w^{(K)}$ for each $x \in \mathsf X$. Hence, we have
\begin{align*}
 \mathcal R^\pi_x(v) - \mathcal R^\pi_x(u) \leq \int \pi(\diff a|x) \bar{\mathcal R}^{(w,K)}_{x,a}(v - u)  \leq \bar{\mathcal U}_x(v-u) \leq \bar{\mathcal U}_x(\lvert v-u \rvert), \forall x \in \mathsf X.
\end{align*}
Switching $v$ and $u$, we obtain 
\begin{align}
 \lvert \mathcal R^\pi_x(v) - \mathcal R^\pi_x(u) \rvert \leq \bar{\mathcal U}_x(\lvert v-u \rvert) \leq \lVert v - u\rVert_{w'} \bar{\mathcal U}_x(w'), \forall x \in \mathsf X. \label{eq:abs}
\end{align}
We consider the following two cases. 

Case I: $w_0(x) + w_0(y) \geq R$ and set $\gamma_0 := \gamma + \frac{2 \bar K}{R} \in (0,1)$ and $\gamma_1 := \frac{2 + \beta R \gamma_0}{2 + \beta R}$ for some $\beta> 0$. It is easy to verify that $\gamma_1 \in (0,1)$. Then \eqref{eq:abs} yields
\begin{align}
 &\lvert \mathcal R^\pi_x(v) - \mathcal R^\pi_x(u) - \mathcal R^\pi_y(v) + \mathcal R^\pi_y(u) \rvert \nonumber \leq \lvert \mathcal R^\pi_x(v) - \mathcal R^\pi_x(u) \rvert + \lvert \mathcal R^\pi_y(v) - \mathcal R^\pi_y(u) \rvert\nonumber \\
\leq & A (2 + \beta \bar{\mathcal U}_x(w_0) + \beta \bar{\mathcal U}_y(w_0)) \leq A (2 + \beta \gamma w_0(x) + \beta \gamma w_0(y) + 2 \beta \bar K) \nonumber\\
\leq & A(2 + \beta \gamma_0 w_0(x) + \gamma_0 w_0(y)) \leq A \gamma_1 (w'(x) + w'(y)). \label{eq:step1}
\end{align}

Case II: $w_0(x) + w_0(y) \leq R$. Hence both $x$ and $y$ are in the subset $\mathsf B$. We define for all $x \in \mathsf B$, $
\tilde{\mathcal R}_x^\pi(v) := \dfrac{1}{1-\alpha} \mathcal R_x^\pi(v) - \dfrac{\alpha}{1-\alpha}\mu(v), \textrm{and } 
\tilde{\mathcal U}_x(v) := \dfrac{1}{1-\alpha} \bar{\mathcal U}_x(v) - \dfrac{\alpha}{1-\alpha}\mu(v).$
Hence, we have $\tilde{\mathcal R}_x^\pi(v) - \tilde{\mathcal R}_x^\pi(u) \leq \tilde{\mathcal U}_x(v - u)$. By Assumption \ue{2}, 
$
 \tilde{\mathcal U}_x(v) - \tilde{\mathcal U}_x(u) \geq 0, \forall v \geq u \in \mathscr B_{1+w_0}, 
$
shows that $\tilde{\mathcal U}_x(\cdot)$ is monotone. It is also easy to verify that $\tilde{\mathcal U}_x(\cdot)$ is translation invariant, centralized and coherent. Hence, 
\begin{align*}
 \lvert \mathcal R^\pi_x(v) - \mathcal R^\pi_x(u) - \mathcal R^\pi_y(v) + \mathcal R^\pi_y(u) \rvert = & (1-\alpha)  \lvert \tilde{\mathcal R}_x^\pi(v) - \tilde{\mathcal R}_x^\pi(u) - \tilde{\mathcal R}_y^\pi(v) + \tilde{\mathcal R}_y^\pi(u) \rvert \\
\leq &  (1-\alpha)  \lvert \tilde{\mathcal R}_x^\pi(v) - \tilde{\mathcal R}_x^\pi(u) \rvert  +  (1-\alpha) \lvert \tilde{\mathcal R}_y^\pi(v) - \tilde{\mathcal R}_y^\pi(u) \rvert \\
\leq & (1-\alpha) \tilde{\mathcal U}_x( \lvert v - u \rvert ) + (1-\alpha) \tilde{\mathcal U}_y( \lvert v - u \rvert ) \\
\leq & 2 A (1-\alpha) + A (1-\alpha) \beta \left( \tilde{\mathcal U}_x( w_0) + \tilde{\mathcal U}_y( w_0) \right).
\end{align*}
Note that since $(1-\alpha)  \tilde{\mathcal U}_x( w_0) \leq \bar{\mathcal U}_x(w_0)$ holds for all $x \in \mathsf B$, we obtain
\begin{align}
\lvert \mathcal R^\pi_x(v) - \mathcal R^\pi_x(u) - \mathcal R^\pi_y(v) + \mathcal R^\pi_y(u) \rvert \leq &  2 A (1-\alpha) + A \beta \left( \bar{\mathcal U}_x( w_0) + \bar{\mathcal U}_y( w_0) \right) \label{eq:w0a}\\ 
\leq & 2 A (1-\alpha) + A \beta (\gamma(w_0(x) + w_0(y)) + 2 \bar K). \nonumber
\end{align}
We select $\beta := \frac{\alpha_0}{\bar K}$ for some $\alpha_0 \in (0, \alpha)$. Setting $\gamma_2 := (1-\alpha + \alpha_0) \vee \gamma \in (0,1)$ yields for all $x \neq y$
\begin{align}
& \lvert \mathcal R^\pi_x(v) - \mathcal R^\pi_x(u) - \mathcal R^\pi_y(v) + \mathcal R^\pi_y(u) \rvert \nonumber \\
\leq & 2 A(1-\alpha + \alpha_0) + A \gamma \beta(w_0(x) + w_0(y)) \leq A \gamma_2(w'(x) + w'(y)).\label{eq:step2}
\end{align}

Hence, setting $\bar \alpha := \gamma_1 \vee \gamma_2 < 1$, \eqref{eq:step1} and \eqref{eq:step2} imply for all $x \neq y$
\begin{align*}
 \lvert \mathcal R^\pi_x(v) - \mathcal R^\pi_x(u) - \mathcal R^\pi_y(v) + \mathcal R^\pi_y(u) \rvert \leq \lVert v -u \rVert_{s,w'} \bar \alpha (w'(x) + w'(y)),
\end{align*}
the required inequality. \quad 
\endproof

Recall that for any unbounded nonnegative $\mathcal B(\mathsf X)$-measurable function $w$ and any real number $R \in \mathbb R$, we define $\mathsf B_w(R) :=  \left\{ x \in \mathsf X | w(x) \leq R \right\}$ 
and $\mathsf B_w^c(R)$ its complementary set.

\paragraph{Proof of Theorem \ref{th:upperenv}} Due to Assumption \ref{ass:lya}, for any $\lambda \in (\gamma_1,1)$, we have
\begin{align*}
\mathcal R_{x,a}(w_1) \leq \lambda w_1(x), \forall x \in \mathsf B_{w_1}^c(A), A:= \frac{K_1}{\lambda - \gamma_1}, a \in \mathsf A(x).
\end{align*}
It implies that for all $x \in \mathsf B_{w_1}^c(A), a \in \mathsf A(x)$, 
\begin{align}
\label{eq:en:0b} & \int_{\mathsf B_{w_1}^c(\lambda w_1(x))} Q_{x,a}(\diff y) \left( e^{w_1(y)-\lambda w_1(x)} - 1\right) \\ 
 \leq & \int_{\mathsf B_{w_1}(\lambda w_1(x))} Q_{x,a}(\diff y) \left( 1- e^{w_1(y)-\lambda w_1(x)} \right).  \nonumber
\end{align}

Taking some $\gamma_2 \in (\lambda^p, 1)$, by the definition of $w_0$ ($:=w_1^p$), we have then
\begin{align}
 \mathsf B_{w_0}^c(\gamma_2 w_0(x)) \subset \mathsf B_{w_1}^c(\lambda w_1(x)), \forall x \in \mathsf X. \label{eq:en:setb}
\end{align}
Indeed, for any $y \in  \mathsf B_{w_0}^c(\gamma_2 w_0(x))$, it follows that $w_0(y) > \gamma_2 w_0(x)$, which is equivalent to
$w_1(y) > (\gamma_2)^{1/p} w_1(x) > \lambda w_1(x)$. Hence, $y \in  \mathsf B_{w_1}^c(\lambda w_1(x))$ as well. Before continuing the proof, we state and prove two lemmas. 
\begin{lemma}\label{lm:en:1b}
For any $\eta \in (0, 1 - \lambda)$, $p \in (0,1)$, $K \geq 0$ and $\gamma_2 \in ((\lambda + \eta)^p,1)$, there exists a constant $R_1 > 0$ such that for all $y \in \mathsf B_{w_0}^c(\gamma_2 w_0(x))$, $x \in \mathsf B_{w_1}^c(R)$ and $R \geq R_1$, 
\begin{align*}
e^{w_0(y) + K + \eta w_1(x)} \left( w_0(y) - \gamma_2 w_0(x)\right) \leq e^{w_1(y) - \lambda w_1(x)} - 1.
\end{align*}
\end{lemma}
\begin{proof}
It is sufficient to show that there exists a constant $R_1 > 0$ satisfying
\begin{align}
w_0(y) + \ln w_0(y) + K + \ln 2 + \eta w_1(x) \leq w_1(y) - \lambda w_1(x) \label{eq:en:3b}
\end{align}
for all $y \in \mathsf B_{w_0}^c(\gamma_2 w_0(x))$, $x \in \mathsf B_{w_1}^c(R)$ and $R \geq R_1 $. Note that for any $p \in (0,1)$ and $\epsilon \in (0,1)$, there exists a constant $D$ (depending on $p$ and $\epsilon$) satisfying $x^p + p \ln x \leq \epsilon x + D, \forall x \geq 1,$ which implies $w_0(x) + \ln w_0(x) \leq \epsilon w_1(x) + D, \forall x \in \mathsf X$. Hence, for all $y \in \mathsf B_{w_0}^c(\gamma_2 w_0(x)) $, we have
\begin{align*}
  & w_1(y) - w_0(y) - \ln w_0(y) - (\lambda +\eta) w_1(x)  \\
 \geq & (1-\epsilon)w_1(y) - (\lambda + \eta) w_1(x) - D 
           \geq \left((1-\epsilon) \gamma_2^{1/p} - \lambda - \eta \right) w_1(x) - D.
\end{align*}
Choosing $\gamma_2 \in ((\lambda + \eta)^p,1) $, $\epsilon < 1 - \frac{\lambda + \eta}{\gamma_2^{1/p}}$ and $R_1 := \frac{D + K + \ln 2}{(1-\epsilon) \gamma_2^{1/p} - \lambda - \eta}$, \eqref{eq:en:3b} holds for all $y \in \mathsf B_{w_0}^c(\gamma_2 w_0(x))$, $x \in  \mathsf B_w^c(R)$ and $R \geq R_1$. \quad
\end{proof}
\begin{lemma}\label{lm:en:2b}
For any $\eta >0$, $p \in (0,1)$, $\gamma_2 \in (\lambda^p, 1)$ and $K \geq 0$, there exists a constant $R_2$ such that for all $y \in \mathsf B_{w_1}(\lambda w_1(x))$, $x \in \mathsf B_{w_1}^c(R)$ and $R \geq R_2$,
\begin{align}
e^{- w_0(y) + \eta w_1(x) - K} \left( \gamma_2 w_0(x) - w_0(y) \right) \geq 1 - e^{w_1(y) - \lambda w_1(x)}. \label{eq:w1:2b}
\end{align}
\end{lemma}

\begin{proof}
It is sufficient to show that $e^{- w_0(y) + \eta w_1(x) - K} \left( \gamma_2 w_0(x) - w_0(y) \right) \geq 1$ under the same condition. Note that there exists a constant $D > 0$ such that
\begin{align*}
\frac{\gamma_2}{\eta} x^p \leq x + D, \forall x \geq 1,
\end{align*}
which yields $- w_0(y) + \eta w_1(x) - K \geq - w_0(y) + \gamma_2 w_0(x) - K - D $ and hence, 
\begin{align*}
e^{- w_0(y) + \eta w_1(x) - K} \left( \gamma_2 w_0(x) - w_0(y) \right) \geq e^{\gamma_2 w_0(x) - w_0(y) - K - D} \left( \gamma_2 w_0(x) - w_0(y) \right).
\end{align*}
For all $y \in \mathsf B_{w_1}(\lambda w_1(x))$, we have $\gamma_2 w_0(x) - w_0(y) \geq (\gamma_2 - \lambda^p) w_0(x).$ 
Hence, $$e^{\gamma_2 w_0(x) - w_0(y) - K - D} \left( \gamma_2 w_0(x) - w_0(y) \right)  \geq e^{(\gamma_2 - \lambda^p) w_0(x) - K - D} (\gamma_2 - \lambda^p) w_0(x).$$ Due to the fact that $g(x) = e^x \cdot x $ is an increasing function on $\mathbb R_+$, we can choose $\tilde R_2 > 0$ such that $e^{\tilde R_2} \cdot \tilde R_2 = e^{K+ D}$. Hence, we have for all $y \in \mathsf B_{w_1}(\lambda w_1(x))$, $x \in \mathsf B_{w_0}^c(\tilde R)$ and $\tilde R \geq \tilde R_2$, $e^{- w_0(y) + \eta w_1(x) - K} \left( \gamma_2 w_0(x) - w_0(y) \right) \geq 1$ holds. Finally, setting $R_2 = \tilde R_2^{1/p}$, the assertion is obtained. \quad
\end{proof}

Hence, by Lemma \ref{lm:en:1b} and \ref{lm:en:2b}, for all $x \in \mathsf B_{w_1}^c(R_1 \vee R_2 \vee A)$ and $a \in \mathsf A(x)$,
\begin{align*}
&\int_{\mathsf B^c_{w_0}(\gamma_2 w_0(x))} Q_{x,a}(\diff y)  e^{w_0(y) + K + \eta w_1(x)} \left( w_0(y) - \gamma_2 w_0(x)\right) \\ \textrm{(Lemma \ref{lm:en:1b})} \quad \leq & \int_{\mathsf B^c_{w_0}(\gamma_2 w_0(x))} Q_{x,a}(\diff y) \left( e^{w_1(y) -\lambda w_1(x)} -1  \right) \\
 \textrm{\eqref{eq:en:setb}} \quad \leq & \int_{\mathsf B^c_{w_1}(\lambda w_1(x))} Q_{x,a}(\diff y) \left( e^{w_1(y) -\lambda w_1(x)} -1  \right) \\
 \textrm{\eqref{eq:en:0b}} \quad \leq & \int_{\mathsf B_{w_1}(\lambda w_1(x))} Q_{x,a}(\diff y) \left( 1- e^{w_1(y) -\lambda w_1(x)}  \right) \\ \textrm{(Lemma \ref{lm:en:2b})} \quad \leq & \int_{\mathsf B_{w_1}( \lambda w_1(x))} Q_{x,a}(\diff y) e^{- w_0(y) + \eta w_1(x) - K} \left( \gamma_2 w_0(x) - w_0(y) \right) \\
 \textrm{\eqref{eq:en:setb}} \quad \leq & \int_{\mathsf B_{w_0}( \gamma_2 w_0(x))} Q_{x,a}(\diff y) e^{- w_0(y) + \eta w_1(x) - K} \left( \gamma_2 w_0(x) - w_0(y) \right) ,
\end{align*}
which implies that for all $f \in \mathscr B_{w_0}$ satisfying $\lvert f \rvert \leq w_0 + K$,
\begin{align}
 \int Q_{x,a}(\diff y) e^{f(y)} \left( w_0(y) - \gamma_2 w_0(x)\right) \leq 0, \forall x \in \mathsf B_{w_1}^c(R_1 \vee R_2 \vee A). \label{eq:gamma2}
\end{align}
Finally, for all $x \in \mathsf B_{w_1}(R_1 \vee R_2 \vee A)$ and $a \in \mathsf A(x)$ and $f \in \mathscr B_{w_0}$ satisfying $\lvert f \rvert \leq w_0 + K$, 
\begin{align*}
\frac{Q_{x,a} [e^{f} w_0]}{Q_{x,a}[e^f]} \leq \frac{Q_{x,a} [e^{w_0 + K} w_0]}{Q_{x,a}[e^{-w_0 - K}]} \leq e^{2K} Q_{x,a} [e^{w_0} w_0] \cdot Q_{x,a}[e^{w_0}]
\end{align*}
Using the fact that there exists some constant $D > 0$ satisfying
\begin{align*}
x^p + p \ln x \leq x + D, \forall x \geq 1,
\end{align*}
we obtain that $Q_{x,a} [e^{w_0} w_0] \leq e^D Q_{x,a} (e^{w_1})$ which is upper bounded on $\mathsf B_{w_1}(R_1 \vee R_2 \vee A)$. Hence, there exists a $K_2 > 0$ such that for all $f \in \mathscr B_{w_0}$ satisfying $\lvert f \rvert \leq w_0 + K$,
\begin{align*}
 \frac{Q_{x,a} [e^{f} w_0]}{Q_{x,a}[e^f]}  \leq K_2, \forall x \in  \mathsf B_{w_1}(R_1 \vee R_2 \vee A),
\end{align*}
which together with \eqref{eq:gamma2} implies the required inequality. \quad \endproof


\small

\begin{thebibliography}{49}
\providecommand{\natexlab}[1]{#1}
\providecommand{\url}[1]{\texttt{#1}}
\expandafter\ifx\csname urlstyle\endcsname\relax
  \providecommand{\doi}[1]{doi: #1}\else
  \providecommand{\doi}{doi: \begingroup \urlstyle{rm}\Url}\fi

\bibitem[Arapostathis et~al.(1993)Arapostathis, Borkar,
  Fern\'{a}ndez-Gaucherand, Ghosh, and Marcus]{arapostathis1993discrete}
A.~Arapostathis, V.S. Borkar, E.~Fern\'{a}ndez-Gaucherand, M.K. Ghosh, and S.I.
  Marcus.
\newblock {Discrete-time controlled Markov processes with average cost
  criterion: a survey}.
\newblock \emph{SIAM Journal on Control and Optimization}, 31\penalty0
  (2):\penalty0 282--344, 1993.

\bibitem[Artzner et~al.(1999)Artzner, Delbaen, Eber, and
  Heath]{artzner1999coherent}
P.~Artzner, F.~Delbaen, J.M. Eber, and D.~Heath.
\newblock {Coherent measures of risk}.
\newblock \emph{Mathematical Finance}, 9\penalty0 (3):\penalty0 203--228, 1999.

\bibitem[Avila-Godoy and Fern\'{a}ndez-Gaucherand(1998)]{avila1998controlled}
G.~Avila-Godoy and E.~Fern\'{a}ndez-Gaucherand.
\newblock {Controlled {M}arkov chains with exponential risk-sensitive criteria:
  modularity, structured policies and applications}.
\newblock In \emph{{Proceedings of the 37th IEEE Conference on Decision and
  Control}}, pages 778--783, 1998.

\bibitem[B\"{a}uerle and Rieder(2013)]{bauerle2013more}
N.~B\"{a}uerle and U.~Rieder.
\newblock {More risk-sensitive Markov decision processes}.
\newblock \emph{Mathematics of Operations Research}, 39, 2013.

\bibitem[Borkar and Meyn(2002)]{borkar2002risk}
V.~S. Borkar and S.P. Meyn.
\newblock {Risk-sensitive optimal control for {M}arkov decision processes with
  monotone cost}.
\newblock \emph{Mathematics of Operations Research}, pages 192--209, 2002.

\bibitem[Cavazos-Cadena(2010)]{cavazos2010optimality}
R.~Cavazos-Cadena.
\newblock {Optimality equations and inequalities in a class of risk-sensitive
  average cost {M}arkov decision chains}.
\newblock \emph{Mathematical Methods of Operations Research}, 71\penalty0
  (1):\penalty0 47--84, 2010.

\bibitem[\c{C}avu\c{s} and Ruszczy\'{n}ski(2013)]{Cavus2014}
\"O. \c{C}avu\c{s} and A.~Ruszczy\'{n}ski.
\newblock {Risk-averse control of undiscounted transient {M}arkov models}.
\newblock \emph{SIAM Journal on Control and Optimization}, 52\penalty0
  (6):\penalty0 3935--3966, 2013.

\bibitem[Cheridito and Li(2009)]{cheridito2009risk}
P.~Cheridito and T.~Li.
\newblock {Risk measures on {O}rlicz hearts}.
\newblock \emph{Mathematical Finance}, 19\penalty0 (2):\penalty0 189--214,
  2009.

\bibitem[Chung and Sobel(1987)]{chung1987discounted}
K.J. Chung and M.J. Sobel.
\newblock {Discounted {MDP}s: distribution functions and exponential utility
  maximization}.
\newblock \emph{SIAM Journal on Control and Optimization}, 25:\penalty0 49,
  1987.

\bibitem[Coraluppi and Marcus(2000)]{coraluppi2000mixed}
S.P. Coraluppi and S.I. Marcus.
\newblock {Mixed risk-neutral/minimax control of discrete-time, finite-state
  {M}arkov decision processes}.
\newblock \emph{IEEE Transactions on Automatic Control}, 45\penalty0
  (3):\penalty0 528--532, 2000.

\bibitem[Del~Moral et~al.(2003)Del~Moral, Ledoux, and
  Miclo]{del2003contraction}
P.~Del~Moral, M.~Ledoux, and L.~Miclo.
\newblock On contraction properties of {M}arkov kernels.
\newblock \emph{Probability Theory and Related Fields}, 126\penalty0
  (3):\penalty0 395--420, 2003.

\bibitem[Delbaen(2000)]{Delbaen_2000}
F.~Delbaen.
\newblock {Coherent risk measures on general probability spaces}.
\newblock \emph{Advances in Finance and Stochastics Essays in Honour of Dieter
  Sondermann}, pages 1--37, 2000.

\bibitem[{Di Masi} and Stettner(2000)]{masi2000infinite}
G.B. {Di Masi} and {L}. Stettner.
\newblock {Infinite horizon risk sensitive control of discrete time {M}arkov
  processes with small risk}.
\newblock \emph{Systems \& control letters}, 40\penalty0 (1):\penalty0 15--20,
  2000.

\bibitem[{Di Masi} and Stettner(2008)]{di2008infinite}
G.B. {Di Masi} and L.~Stettner.
\newblock {Infinite horizon risk sensitive control of discrete time {M}arkov
  processes under minorization property}.
\newblock \emph{SIAM Journal on Control and Optimization}, 46\penalty0
  (1):\penalty0 231, 2008.

\bibitem[Douc et~al.(2004)Douc, Moulines, and Rosenthal]{douc2004quantitative}
R.~Douc, E.~Moulines, and J.S. Rosenthal.
\newblock Quantitative bounds on convergence of time-inhomogeneous {M}arkov
  chains.
\newblock \emph{Annals of Applied Probability}, pages 1643--1665, 2004.

\bibitem[Douc et~al.(2009)Douc, Fort, Moulines, and
  Priouret]{douc2009forgetting}
R.~Douc, G.~Fort, E.~Moulines, and P.~Priouret.
\newblock {Forgetting the initial distribution for hidden Markov models}.
\newblock \emph{Stochastic Processes and Their Applications}, 119\penalty0
  (4):\penalty0 1235--1256, 2009.

\bibitem[Filar et~al.(1989)Filar, Kallenberg, and Lee]{filar1989variance}
J.A. Filar, LCM Kallenberg, and H.M. Lee.
\newblock {Variance-penalized {M}arkov decision processes}.
\newblock \emph{Mathematics of Operations Research}, pages 147--161, 1989.

\bibitem[Fleming and Hern\'{a}ndez-Hern\'{a}ndez(1997)]{fleming1997risk}
W.H. Fleming and D.~Hern\'{a}ndez-Hern\'{a}ndez.
\newblock {Risk-sensitive control of finite state machines on an infinite
  horizon {I}}.
\newblock \emph{SIAM Journal on Control and Optimization}, 35\penalty0
  (5):\penalty0 1790--1810, 1997.

\bibitem[F\"{o}llmer and Schied(2002)]{follmer2002convex}
H.~F\"{o}llmer and A.~Schied.
\newblock {Convex measures of risk and trading constraints}.
\newblock \emph{Finance and Stochastics}, 6\penalty0 (4):\penalty0 429--447,
  2002.

\bibitem[F\"{o}llmer and Schied(2004)]{follmer2004stochastic}
H.~F\"{o}llmer and A.~Schied.
\newblock \emph{{Stochastic Finance}}.
\newblock Walter de Gruyter \& Co., Berlin, 2004.
\newblock Extended edition.

\bibitem[Gaubert and Gunawardena(2004)]{gaubert2004perron}
S.~Gaubert and J.~Gunawardena.
\newblock {The {P}erron-{F}robenius theorem for homogeneous, monotone
  functions}.
\newblock \emph{Transactions American Mathematical Society}, 356\penalty0
  (12):\penalty0 4931--4950, 2004.

\bibitem[Hairer and Mattingly(2011)]{hairer2011yet}
M.~Hairer and J.C. Mattingly.
\newblock {Yet another look at {H}arris' ergodic theorem for {M}arkov chains}.
\newblock In \emph{{Seminar on Stochastic Analysis, Random Fields and
  Applications VI}}, pages 109--117. Springer, 2011.

\bibitem[Hern\'{a}ndez-Hern\'{a}ndez and Marcus(1996)]{hernandez1996risk}
D.~Hern\'{a}ndez-Hern\'{a}ndez and S.I. Marcus.
\newblock {Risk sensitive control of {M}arkov processes in countable state
  space}.
\newblock \emph{Systems \& Control Letters}, 29\penalty0 (3):\penalty0
  147--155, 1996.

\bibitem[Hern\'{a}ndez-Lerma(1989)]{hernandez1989adaptive}
O.~Hern\'{a}ndez-Lerma.
\newblock \emph{{Adaptive {M}arkov Control Processes}}.
\newblock Springer, 1989.

\bibitem[Hern\'{a}ndez-Lerma and Lasserre(1996)]{hernandez1996discrete}
O.~Hern\'{a}ndez-Lerma and J.B. Lasserre.
\newblock \emph{{Discrete-time {M}arkov Control Processes: Basic Optimality
  Criteria}}.
\newblock Springer, 1996.

\bibitem[Hern\'{a}ndez-Lerma and Lasserre(1999)]{hernandez1999further}
O.~Hern\'{a}ndez-Lerma and J.B. Lasserre.
\newblock \emph{{Further Topics on Discrete-Time {M}arkov Control Processes}}.
\newblock Springer Verlag, 1999.

\bibitem[Howard and Matheson(1972)]{howard1972risk}
R.A. Howard and J.E. Matheson.
\newblock {Risk-sensitive {M}arkov decision processes}.
\newblock \emph{Management Science}, 18\penalty0 (7):\penalty0 356--369, 1972.

\bibitem[Iyengar(2005)]{iyengar2005robust}
G.N. Iyengar.
\newblock {Robust dynamic programming}.
\newblock \emph{Mathematics of Operations Research}, pages 257--280, 2005.

\bibitem[Ja\'{s}kiewicz(2007)]{jaskiewicz2007}
A.~Ja\'{s}kiewicz.
\newblock Average optimality for risk-sensitive control with general state
  space.
\newblock \emph{The Annals of Applied Probability}, 17\penalty0 (2):\penalty0
  654--675, 04 2007.

\bibitem[Kontoyiannis and Meyn(2003)]{kontoyiannis2003spectral}
I.~Kontoyiannis and S.P. Meyn.
\newblock {Spectral theory and limit theorems for geometrically ergodic
  {M}arkov processes}.
\newblock \emph{Annals of Applied Probability}, pages 304--362, 2003.

\bibitem[Kontoyiannis and Meyn(2005)]{kontoyiannis2005large}
I.~Kontoyiannis and S.P. Meyn.
\newblock {Large deviations asymptotics and the spectral theory of
  multiplicatively regular {M}arkov processes}.
\newblock \emph{Electron. J. Probab.}, 10\penalty0 (3):\penalty0 61--123, 2005.

\bibitem[Kupper and Schachermayer(209)]{Kupper2009}
M.~Kupper and W.~Schachermayer.
\newblock Representation results for law invariant time consistent functions.
\newblock \emph{Mathematics and Financial Economics}, 2, 209.

\bibitem[Ledoux(2001)]{ledoux2001concentration}
M.~Ledoux.
\newblock \emph{{The Concentration of Measure Phenomenon}}.
\newblock American Mathematical Society, 2001.

\bibitem[Levin et~al.(2009)Levin, Peres, and Wilmer]{levin2009markov}
D.A. Levin, Y.~Peres, and E.L. Wilmer.
\newblock \emph{Markov Chains and Mixing Times}.
\newblock American Mathematical Society, 2009.

\bibitem[Marcus et~al.(1997)Marcus, Fern\'{a}ndez-Gaucherand,
  Hern\'{a}ndez-Hernandez, Coraluppi, and Fard]{marcus1997risk}
S.I. Marcus, E.~Fern\'{a}ndez-Gaucherand, D.~Hern\'{a}ndez-Hernandez,
  S.~Coraluppi, and P.~Fard.
\newblock {Risk sensitive {M}arkov decision processes}.
\newblock \emph{Progress in Systems and Control Theory}, 22:\penalty0 263--280,
  1997.

\bibitem[Meyn and Tweedie(1993)]{meyn1993markov}
S.P. Meyn and R.L. Tweedie.
\newblock \emph{{{M}arkov Chains and Stochastic Stability}}.
\newblock Springer-Verlag London Ltd., London, 1993.

\bibitem[Nilim and El~Ghaoui(2005)]{nilim2005robust}
A.~Nilim and L.~El~Ghaoui.
\newblock Robust control of markov decision processes with uncertain transition
  matrices.
\newblock \emph{Operations Research}, 53\penalty0 (5):\penalty0 780--798, 2005.

\bibitem[Ogryczak and Ruszczy\'{n}ski(1999)]{ogryczak1999stochastic}
W.~Ogryczak and A.~Ruszczy\'{n}ski.
\newblock {From stochastic dominance to mean-risk models: Semideviations as
  risk measures}.
\newblock \emph{European Journal of Operational Research}, 116\penalty0
  (1):\penalty0 33--50, 1999.

\bibitem[Puterman(1994)]{puterman1994markov}
M.L. Puterman.
\newblock \emph{{{M}arkov Decision Processes: Discrete Stochastic Dynamic
  Programming}}.
\newblock John Wiley \& Sons, Inc., 1994.

\bibitem[Rockafellar and Uryasev(2000)]{rockafellar2000Optconval}
R.T. Rockafellar and S.~Uryasev.
\newblock {Optimization of conditional value-at-risk}.
\newblock \emph{Journal of Risk}, 2:\penalty0 21--42, 2000.

\bibitem[Roorda and Schumacher(2007)]{roorda2007time}
B.~Roorda and J.M. Schumacher.
\newblock Time consistency conditions for acceptability measures, with an
  application to tail value at risk.
\newblock \emph{Insurance: Mathematics and Economics}, 40\penalty0
  (2):\penalty0 209--230, 2007.

\bibitem[Rusz\-czy\'{n}ski(2010)]{ruszczynski2010risk}
A.~Rusz\-czy\'{n}ski.
\newblock {Risk-averse dynamic programming for {M}arkov decision processes}.
\newblock \emph{Mathematical Programming}, pages 1--27, 2010.

\bibitem[Ruszczynski and Shapiro(2006)]{ruszczynski2006optimization}
A.~Ruszczynski and A.~Shapiro.
\newblock Optimization of convex risk functions.
\newblock \emph{Mathematics of operations research}, 31\penalty0 (3):\penalty0
  433--452, 2006.

\bibitem[Schied et~al.(2009)Schied, F\"{o}llmer, and Weber]{schied2009robust}
A.~Schied, H.~F\"{o}llmer, and S.~Weber.
\newblock {Robust preferences and robust portfolio choice}.
\newblock \emph{Handbook of Numerical Analysis}, 15:\penalty0 29--87, 2009.

\bibitem[Shen et~al.(2013)Shen, Stannat, and Obermayer]{Shen2013}
Y.~Shen, W.~Stannat, and K.~Obermayer.
\newblock {Risk-sensitive {M}arkov Control Processes}.
\newblock \emph{SIAM Journal on Control and Optimization}, 51\penalty0
  (5):\penalty0 3652--3672, 2013.

\bibitem[Shen et~al.(2014)Shen, Tobia, Sommer, and Obermayer]{shen2013b}
Y.~Shen, M.J. Tobia, T.~Sommer, and K.~Obermayer.
\newblock {Risk-sensitive reinforcement learning}.
\newblock \emph{Neural Computation}, 26\penalty0 (7):\penalty0 1298--1328, July
  2014.

\bibitem[Sobel(1982)]{sobel1982variance}
M.J. Sobel.
\newblock {The variance of discounted {M}arkov decision processes}.
\newblock \emph{Journal of Applied Probability}, pages 794--802, 1982.

\bibitem[Svindland(2009)]{svindland2009subgradients}
G.~Svindland.
\newblock {Subgradients of law-invariant convex risk measures on {$L^1$}}.
\newblock \emph{Statistics \& Decisions}, 27\penalty0 (2):\penalty0 169--199,
  2009.

\bibitem[Tversky and Kahneman(1992)]{tversky1992advances}
A.~Tversky and D.~Kahneman.
\newblock {Advances in prospect theory: cumulative representation of
  uncertainty}.
\newblock \emph{Journal of Risk and Uncertainty}, 5\penalty0 (4):\penalty0
  297--323, 1992.

\end{thebibliography}
\end{document}